\documentclass[a4paper,12pt,reqno]{amsart}

\usepackage{amsmath}
\usepackage{tikz}
\usepackage{mathdots}
\usepackage{yhmath}
\usepackage{cancel}
\usepackage{color}
\usepackage{siunitx}
\usepackage{array}
\usepackage{multirow}
\usepackage{amssymb}
\usepackage{gensymb}
\usepackage{tabularx}
\usepackage{booktabs}
\usetikzlibrary{fadings}
\usetikzlibrary{patterns}
\usetikzlibrary{shadows.blur}
\usetikzlibrary{shapes}

\usepackage{mathrsfs}

\RequirePackage{latexsym} \RequirePackage{amsthm}
\usepackage{enumerate}
\usepackage{dsfont}
\usepackage{mathrsfs}
\numberwithin{equation}{section}

\newtheorem{conjecture} {\sc  Conjecture\rm} [section]

\newtheorem{preremark}[conjecture]{Remark}
\newenvironment{remark}%
  {\begin{preremark}}{\end{preremark}}

\newtheorem{predefinition}[conjecture]{Definition}
\newenvironment{definition}%
  {\begin{predefinition}}{\end{predefinition}}

\newtheorem{prelemma}[conjecture]{Lemma}
\newenvironment{lemma}%
 {\begin{prelemma}}{\end{prelemma}}
\newtheorem{preproposition}[conjecture]{Proposition}
\newenvironment{proposition}%
 {\begin{preproposition}}{\end{preproposition}}

\newtheorem{precorollary}[conjecture]{Corollary}
\newenvironment{corollary}%
  {\begin{precorollary}}{\end{precorollary}}

\newtheorem{pretheorem}[conjecture]{Theorem}
\newenvironment{theorem}%
 {\begin{pretheorem}}{\end{pretheorem}}

\begin{document}

\title[Applications of P-functions to Fully Nonlinear Elliptic equations]{\bf
Applications of P-functions to Fully Nonlinear Elliptic equations: Gradient Estimates and Rigidity Results}

\author{Dimitrios Gazoulis}
\address{Department of Mathematics, University of Athens (EKPA), Panepistemiopolis, 15784 Athens, Greece}
\email{dgazoulis@math.uoa.gr}
\date{}

\maketitle

\begin{abstract} We introduce the notion of $ P -$functions for fully nonlinear equations and establish a general criterion for obtaining such quantities for this class of equations. Some applications are gradient bounds, De Giorgi-type properties of entire solutions and rigidity results. In particular, we establish a pointwise gradient bound and a rigidity result for Pucci's equations. This pointwise gradient bound generalizes the Modica inequality in the case of fully nonlinear elliptic equations. Furthermore, we prove Harnack-type inequalities and local estimates for the gradient of solutions. In addition, we consider such quantities for higher order nonlinear equations and for equations of order greater than two we obtain Liouville-type theorems and pointwise estimates for the Laplacian.

\end{abstract}



\section{Introduction}

In this work we introduce the notion of \textit{$ P- $ function} for fully nonlinear partial differential equations or differential inequalities and we incorporate the ``\textit{$ P- $ function} technique'' in a general setting. This abstract setting allow us to obtain many applications by only determining an example of a \textit{$ P- $function}. Some of these applications are gradient bounds for entire solutions, Harnack -type inequalities for the gradient of solutions and rigidity results.

The structure of this paper is as follows. In section 2 we define the notion of \textit{$ P- $ functions} and study such quantities for fully nonlinear elliptic equations. We provide two general criteria for obtaining \textit{$ P- $ functions} for this class of equations. Different examples of such quantities may give various types of gradient bounds for solutions of a particular equation. 

In section 3, we prove an abstract pointwise estimate, i.e. Theorem \ref{ThmGradientBoundforDeltau=F}, for a class of \textit{$ P- $ functions} that are associated to any given fully nonlinear equation. This pointwise estimate is in fact a gradient bound for entire solutions in a wide variety of fully nonlinear elliptic equations and some examples of such gradient bounds are given. As an application, we prove a pointwise gradient bound for entire solutions of Pucci's equations that generalizes the Modica inequality. 

Additional consequences are also illustrated in section 4 where we establish a Liouville-type theorem and a De Giorgi-type property for entire solutions. One application is Theorem \ref{theoremRigidityForFullyNonlinear}, that is an abstract rigidity result for entire solutions of fully nonlinear elliptic equations such as Pucci's equations. This result also recovers as particular cases some of the classical results of J. Serrin in \cite{Serrin}. 

In section 5 we establish a Harnack-type inequality and local pointwise estimates for the gradient of solutions to  fully nonlinear elliptic equations, such as for the Monge-Ampère equation.

Moreover, in section 6, we study such quantities for nonlinear equations of order greater than two together with some applications. For example, we establish an a priori bound for the Laplacian and pointwise estimates through the mean value properties for higher order equations. Also, some Liouville-type properties can be established for nonlinear equations of order greater than two. In this setting, one can obtain many other types of bounds for any order of derivatives, assuming a $ C^{k, \alpha} $ a priori estimate and provided that we have an appropriate \textit{$ P -$ function} related to the respective equation.

We will now briefly discuss some of the most important contributions on the ``\textit{$ P- $ function} technique'' and it's applications. Perhaps the most well-known example is $$ P(u,x) = \frac{1}{2} | \nabla u |^2 - W(u) $$
that is related to the Allen-Cahn equation
\begin{equation}\label{AllenCahnEquation}
\Delta u = W'(u) \;\;,\; u : \Omega \subset \mathbb{R}^n \rightarrow \mathbb{R},
\end{equation}
and Modica in \cite{Modica} proved the well-known gradient bound 
\begin{equation}\label{ModicaInequality}
\frac{1}{2} | \nabla u|^2 \leq W(u)
\end{equation}
for every bounded entire solution of \eqref{AllenCahnEquation}.

Later, Caffarelli et al in \cite{CGS} generalized this gradient bound for a class of varational quasi-linear equations and proved Liouville-type and De Giorgi-type properties for a particular choise of \textit{$ P- $function} related to the equation $ div( \Phi'(| \nabla u |^2 )\nabla u ) = F'(u) $. This bound was extended for anisotropic partial differential equations and other general types of equations in \cite{CDFGV,FV,FV2}. The gradient bound \eqref{ModicaInequality} also holds in unbounded domains with nonnegative mean curvature as proved in \cite{FV3}.

\textit{$ P- $functions} had been already studied by Sperb in \cite{Sperb}, Payne and Philippin in \cite{PP,PP2} who studied other types of quasilinear equations for the form $ div \: (A(u,| \nabla u |^2)\nabla u) = B(u,| \nabla u |^2) $, which are not necessarily Euler-Lagrange equations of an elliptic integrand. They derived maximum principles for some appropriate \textit{$ P- $functions}. Due to the greater generality, however, the relevant $ P $ and the conditions under which satisfies an elliptic differential inequality are rather implicitly given while in \cite{CGS,DG} are given explicitly. \textit{$ P- $functions} can also be utilized to derive lower bounds for eigenvalue problems (see \cite{Sperb}).

There are many other important contributions related to the \textit{$ P- $function} technique that can be found in \cite{BG,CFV,FV3,FWX,GL,GS} to cite a few.

One additional important application is in \cite{AAC}, where they established that the monotonicity assumption $ u_{x_n} >0 $, that is also stated in the De Giorgi's conjecture, does in fact imply the local minimality of $ u $. Such implication is by no means trivial and it is based on the construction of a so-called \textit{calibration} associated to the energy functional. Such notion is intimately connected to the theory of null-Lagrangians, see \cite{GH}, chapter 1 and chapter 4, section 2.4. In Theorem 4.4 in \cite{AAC}, they carry out the construction of the appropriate calibration for general integrands of the calculus of variations and such construction relies explicitly on the \textit{$ P- $ function}. 

Last but not least, there are applications such as gradient bounds similar to \eqref{ModicaInequality} and Liouville-type properties for vector equations. To be more precise, in Theorem 3.5 in \cite{Smyrnelis}, is a gradient bound is proved for the Ginzburg-Landau system of equations.

$ \\ $

\section{$ P- $functions for Fully Nonlinear Elliptic equations}

We begin by defining the notion of $ P-$function 

\begin{definition}\label{DefinitionP-function} 
Let $ u : \Omega \subset \mathbb{R}^n \rightarrow \mathbb{R}^d $ be a smooth solution or subsolution of
\begin{equation}\label{FullyNonlinearPDE}
F(x,u, \nabla u, ..., \nabla^m u )= 0
\end{equation}
where $ F $ is a real valued function defined on $ \Omega \times \mathbb{R} \times \mathbb{R}^n \times ... \times \mathbb{R}^{n^m} $.

We say that $ P = P(x,u, \nabla u , ..., \nabla^{m-1} u) $ is a $ P -$function of \eqref{FullyNonlinearPDE} if there exists an elliptic operator $ L $ and a non negative function $ \mu = \mu(x) \geq 0 $
\begin{equation}\label{EllipticOperatorLP<=0}
\begin{gathered}
L = - \sum_{i,j=1}^n a_{ij} \partial_{x_i x_j} + \sum_{i=1}^n b_i \partial_{x_i} +c \;\;,\; \textrm{with} \;\: c \geq 0, \\
\textrm{such that} \;\: \mu L \: P \leq 0 \;\;,\; \textrm{in} \;\: \Omega.
\end{gathered}
\end{equation}
\end{definition}
$ \\ $

An immediate corollary is that any $ P -$function related to an equation or to a differential inequality attains its maximum at the boundary $ \partial \Omega $ or at a point $ x \in \Omega $ such that $ \mu (x) =0 $.

$ \\ $

We initially state as a direct consequence a strong maximum principle that holds in general (see Theorem 2.2 in \cite{CGS} or Theorem 4.7 in \cite{DG}). $ \\ $

\begin{theorem}\label{AbstractConsequenceofPfunction1}
Let $ u $ be a smooth solution or subsolution of
\begin{equation}\label{GeneralFullyNonlinearPDE}
\begin{gathered}
F(x,u, \nabla u, ..., \nabla^m u) =0 \;\;\;,\;\; u : \Omega \rightarrow \mathbb{R}^d \\
\textrm{where} \;\: \Omega \;\: \textrm{is a connected, bounded subset of} \;\: \mathbb{R}^n,
\end{gathered}
\end{equation}
such that $ \inf_{\overline{\Omega}} g( \nabla^k u )>0 $ for some $ g :\mathbb{R}^{n^k \times d} \rightarrow [0, + \infty) \;,\; k \in \lbrace 1,..., m-1 \rbrace $ and suppose that $ P = P(x,u, \nabla u,..., \nabla^{m-1} u) $ is a $ P- $function of \eqref{GeneralFullyNonlinearPDE} with $ \mu = \mu ( g( \nabla^k u )) \;, \: \mu(t) >0 \;, \: \forall \: t>0 $.

If there exists $ x_0 \in \Omega $ such that
\begin{equation}\label{GeneralPattainsSup}
P(x_0,u(x_0),..., \nabla^{m-1} u(x_0)) = \sup_{\Omega} P(x,u,..., \nabla^{m-1} u) \geq 0 ,
\end{equation}
then $ P(x,u, \nabla u,..., \nabla^{m-1} u) $ is constant in $ \Omega $.
\end{theorem}

\begin{proof}
The proof is an immediate consequence of the strong maximum principle since $ \\ \mu ( g( \nabla^k u )) >0 $ in $ \Omega $.
\end{proof}
$ \\ $

The most common choice of $ g $ in Theorem \ref{AbstractConsequenceofPfunction1} above is the Euclidean norm. For example, if $ k=1 $, $ g( \nabla u) =| \nabla u | .$ If $ \mu >0 \;,\: \forall \: t \geq 0 $, then the assumption $ \inf_{\overline{\Omega}} g( \nabla^k u )>0 $ is dismissed. $ \\ $

\begin{remark}
The constancy of $ P- $functions with a particular form hides geometric information on the level sets $ \lbrace x \in \mathbb{R}^n \: | \: u(x)=t \rbrace $ of the solution $ u $, such as the property of being surfaces of zero mean curvature (see Proposition 4.11 in \cite{DG}). 
\end{remark}

$ \\ $

We now focus on fully nonlinear elliptic equations. Let $ u : \Omega \subset \mathbb{R}^n \rightarrow \mathbb{R} $ be a smooth solution of
\begin{equation}\label{FullyNonlinearEllipticEq}
F(x,u, \nabla u, \nabla^2 u) = 0 ,
\end{equation}
where $ F : \Omega \times \mathbb{R} \times \mathbb{R}^n \times \mathbb{R}^{n \times n} \rightarrow \mathbb{R} $ 
satisfies the ellipticity condition
\begin{equation}\label{EllipticityConditionForFullyNonlinearElEq}
\lambda | \xi |^2 \leq \sum_{i,j} F_{a_{ij}} (x,u, \nabla u, \nabla^2 u) \xi_i \xi_j \leq \Lambda | \xi |^2 \;\;\;,\; \forall \: x \in \Omega\;,\; \forall \: \xi \in \mathbb{R}^n .
\end{equation}
Here we use the notation $ F= F(x,s,q,A) \;,\; s \in \mathbb{R}\;,\; q \in \mathbb{R}^n \;,\; A \in \mathbb{R}^{n \times n} $ and $ F_{a_{ij}}= \dfrac{\partial F}{\partial a_{ij}} . $
Some important examples of fully nonlinear elliptic equations are $ \\ \\ $
\textbf{(1)} Pucci's equations. $ \\ $
We introduce the Pucci's extremal operators. Let $ 0 < \lambda \leq \Lambda $ and $ A \in \mathcal{S} $, where $ \mathcal{S} $ is the class of symmetric $ n \times n $ matrices, we define
\begin{equation}\label{Pucci'sOperators}
\begin{gathered}
\mathscr{M}^- (A, \lambda, \Lambda) = \mathscr{M}^- (A) = \lambda \sum_{e_i >0} e_i + \Lambda \sum_{e_i < 0} e_i \\
\mathscr{M}^+ (A, \lambda, \Lambda) = \mathscr{M}^+ (A) = \lambda \sum_{e_i <0} e_i + \Lambda \sum_{e_i > 0} e_i
\end{gathered}
\end{equation}
where $ e_i = e_i(A) $ are the eigenvalues of $ A . $ It holds that $ \mathscr{M}^- $ and $ \mathscr{M}^+ $ are uniformly elliptic with ellipticity constants $ \lambda $ and $ n \Lambda $ (see \cite{CaffarelliCabre}).

Pucci's equations are
\begin{equation}\label{Pucci'sEquationsM-}
\mathscr{M}^- ( \nabla^2 u) = f(x,u)
\end{equation}
and
\begin{equation}\label{Pucci'sEquationsM+}
\mathscr{M}^+ ( \nabla^2 u) = f(x,u)
\end{equation}
respectively.
$ \\ \\ $
\textbf{(2)} Monge-Ampère's equation
\begin{align*}
det ( \nabla^2 u) = f(x)
\end{align*}
for strictly convex solutions $ u $ and $ f>0. \\ \\ $
\textbf{(3)} Equation of Prescribed Gauss curvature
\begin{align*}
det( \nabla^2 u) = K(x) ( 1 + | \nabla u |^2)^{\frac{n+2}{2}}
\end{align*} 
$ K(x) $ is the Gauss curvature of the graph $ u $ at $ (x,u(x)) $. Again, this equation is elliptic for strictly convex solutions $ u . \\ \\ $
\textbf{(4)} Quasi-Linear equations of the form
\begin{align*}
\sum_{i,j} a_{ij}(\nabla u) u_{x_i x_j} = F(x,u,\nabla u)
\end{align*}
where $ a_{ij} $ satisfy the ellipticity condition $ \lambda | \xi |^2 \leq \sum_{i,j} a_{ij} \xi_i \xi_j \leq \Lambda | \xi |^2 $. The $ p- $Laplace equation 
, the Allen-Cahn equation and the minimal surface equation belong in this class of equations. Such equations are thoroughly studied in \cite{CGS,DG,FV2} among others.
$ \\ $

There are many other examples of fully nonlinear elliptic equations, such as Bellman equation and Isaacs equation (see \cite{CaffarelliCabre}). $ \\ $

We now provide two general criteria for obtaining $ P- $functions. $ \\ $

\begin{lemma}\label{LemmaCriterionPfunctionsFullyNonlinearEq}
Let $ u $ be a smooth solution of \eqref{FullyNonlinearEllipticEq} and $ F $ satisfies \eqref{EllipticityConditionForFullyNonlinearElEq}. Consider the quantity
\begin{equation}\label{LemmaQuantityIforPfunctions}
I:= 2 \lambda | \nabla^2 u|^2 - 2F_u| \nabla u |^2 +  (g'(u) \nabla_q F - 2 \nabla_x F ) \nabla u + \sum_{i,j} F_{a_{ij}} ( g'(u) u_{x_i x_j} + g''(u) u_{x_i} u_{x_j} )
\end{equation}
and assume that $ I \geq 0 $ for some smooth function $ g : \mathbb{R} \rightarrow \mathbb{R} . $ 

Then $ P(u,\nabla u) = | \nabla u |^2 + g(u) $ is a $ P - $function of \eqref{FullyNonlinearEllipticEq}.
\end{lemma}
$ \\ $
\begin{proof}
Assume that $ I \geq 0 $ for some $ g : \mathbb{R} \rightarrow \mathbb{R} $ and let $ P(u,| \nabla u|^2) = | \nabla u |^2 + g(u) $.

We have
\begin{equation}\label{ProofLemmaGeneralCritFulNonEq1}
\begin{gathered}
P_{x_i} = 2 \sum_{k} u_{x_k} u_{x_k x_i} + g'(u) u_{x_i} \\
\Rightarrow P_{x_i x_j} = 2 \sum_{k} (u_{x_k x_j} u_{x_k x_i} + u_{x_k} u_{x_k x_i x_j}) + g''(u) u_{x_i}u_{x_j} + g'(u) u_{x_i x_j} \\
\Rightarrow \sum_{i,j} d_{ij} P_{x_i x_j} = 2 \sum_{i,j,k} (d_{ij} u_{x_k x_j} u_{x_k x_i} + d_{ij} u_{x_k} u_{x_k x_i x_j}) + g''(u) \sum_{i,j} d_{ij} u_{x_i}u_{x_j} \\ + g'(u) \sum_{i,j} d_{ij} u_{x_i x_j} \\ \geq 2 \lambda | \nabla^2 u |^2 + 2  \sum_{i,j,k} d_{ij} u_{x_k} u_{x_k x_i x_j} + \sum_{i,j} d_{ij} g''(u) u_{x_i} u_{x_j} + g'(u) \sum_{i,j} d_{ij} u_{x_i x_j}
\end{gathered}
\end{equation}
where $ d_{ij} = F_{a_{ij}} = \dfrac{\partial F}{\partial a_{ij}} $.  

Differentiating \eqref{FullyNonlinearEllipticEq} over $ x_k $, and then multiplying by $ u_{x_k} $, we obtain
\begin{equation}\label{ProofLemmaGeneralCritFulNonEq2}
\begin{gathered}
F_{x_k} + F_u u_{x_k} + \sum_{m} F_{q_m} u_{x_m x_k} + \sum_{m,l} d_{ml} u_{x_m x_l x_k} = 0 \\ \Rightarrow \sum_{m,l,k} d_{ml} u_k u_{x_m x_l x_k} = - \nabla_x F \nabla u - F_u | \nabla u |^2 - \frac{1}{2} \nabla_q F \nabla_x P + \frac{1}{2} g'(u) \nabla_q F \nabla u 
\end{gathered}
\end{equation}
Therefore the last equation of \eqref{ProofLemmaGeneralCritFulNonEq1} becomes
\begin{equation}\label{ProofLemmaGeneralCritFulNonEq3}
\begin{gathered}
\sum_{i,j} d_{ij} P_{x_i x_j} + \nabla_q F \nabla_x P \geq 2 \lambda | \nabla^2 u |^2 +   g'(u) \nabla_q F \nabla u - 2F_u| \nabla u |^2 \\ - 2 \nabla u \nabla_x F +  \sum_{i,j} d_{ij} g''(u) u_{x_i} u_{x_j}  + g'(u) \sum_{i,j} d_{ij} u_{x_i x_j} \geq 0
\end{gathered}
\end{equation}
\end{proof}
$ \\ $

\begin{lemma}\label{LemmaCriterion2PfunctionsFullyNonlinearEq}
Let $ u $ be a smooth solution of \eqref{FullyNonlinearEllipticEq} and $ F $ satisfies \eqref{EllipticityConditionForFullyNonlinearElEq}. Consider the quantity
\begin{equation}\label{Lemma2QuantityJforPfunctions}
\begin{gathered}
J:= B''(u) \sum_{i,j} F_{a_{ij}} u_{x_i} u_{x_j}  - 2 A'(| \nabla u |^2) F_u| \nabla u |^2  +(B'(u) \nabla_q F - 2 A'(| \nabla u |^2) \nabla_x F) \nabla u \\ + B'(u) \sum_{i,j} F_{a_{ij}} u_{x_i x_j} 
+ \frac{\lambda (B'(u))^2 }{2 A'(| \nabla u |^2)}
\end{gathered}
\end{equation}
for some smooth functions $ A,B : \mathbb{R} \rightarrow \mathbb{R} $ with $ A'>0 \;,\:  A''\geq 0 \;,\: \forall t>0  . $ Assume that $ J \geq 0 $ .

Then $ P(u,\nabla u) = A(| \nabla u |^2) + B(u) $ is a $ P - $function of \eqref{FullyNonlinearEllipticEq}.
\end{lemma}
$ \\ $
\begin{proof}
We argue as in Lemma \ref{LemmaCriterionPfunctionsFullyNonlinearEq} and obtain
\begin{equation}\label{ProofLemmaCrit2Eq1}
\begin{gathered}
\sum_{i,j} d_{ij} P_{x_i x_j} = A''( \sum_{i,j} d_{ij} [| \nabla u |^2]_{x_i} [| \nabla u |^2]_{x_j} ) + 2A' ( \sum_{i,j,k} d_{ij}[u_{x_k x_j}u_{x_k x_i} + u_{x_k} u_{x_k x_i x_j}]) \\ + \sum_{i,j} d_{ij} (B''(u) u_{x_i}u_{x_j} + B'(u) u_{x_ix_j}) \\
\geq 2 A' \lambda | \nabla^2 u |^2 + 2 A' \sum_{i,j,k} d_{ij} u_{x_k} u_{x_k x_i x_j} + B''(u) \sum_{i,j} d_{ij} u_{x_i} u_{x_j} + B' \sum_{i,j} d_{ij}u_{x_i x_j}
\end{gathered}
\end{equation}
by \eqref{EllipticityConditionForFullyNonlinearElEq} and since $ A'' \geq 0 $, where $ d_{ij}= \dfrac{\partial F}{\partial a_{ij}} . $

We also calculate from the first equation of \eqref{ProofLemmaGeneralCritFulNonEq1} and the Cauchy-Schwarz inequality
\begin{equation}\label{ProofLemmaCrit2Eq2}
\begin{gathered}
\sum_i (P_{x_i} -B' u_{x_i})^2 \leq 4 (A')^2 | \nabla u |^2 | \nabla^2 u |^2 \\
\Rightarrow 2 A' | \nabla^2 u |^2 \geq \frac{1}{2| \nabla u |^2 A'}( |\nabla P|^2 - 2 B' \nabla u \nabla P + (B')^2 | \nabla u |^2)
\end{gathered}
\end{equation}
In addition, similarly to \eqref{ProofLemmaGeneralCritFulNonEq2} we have
\begin{equation}\label{ProofLemmaCrit2Eq3}
\sum_{m,l,k} d_{ml} u_{x_k} u_{x_m x_l x_k} = - \nabla_x F \nabla u - F_u | \nabla u |^2 - \frac{1}{2 A'} (\nabla_q F \nabla_x P - B'  \nabla u \nabla_q F)
\end{equation}
We plug \eqref{ProofLemmaCrit2Eq2} and \eqref{ProofLemmaCrit2Eq3} into \eqref{ProofLemmaCrit2Eq1} and thus
\begin{equation}\label{ProofLemmaCrit2Eq4}
\begin{gathered}
\sum_{i,j} d_{ij} P_{x_i x_j} \geq \frac{\lambda}{2| \nabla u |^2 A'}( |\nabla P|^2 - 2 B' \nabla u \nabla P + (B')^2 | \nabla u |^2) +  B''(u) \sum_{i,j} d_{ij} u_{x_i} u_{x_j}  \\ + B' \sum_{i,j} d_{ij}u_{x_i x_j} + 2 A'( - \nabla_x F \nabla u - F_u | \nabla u |^2 - \frac{1}{2 A'} (\nabla_q F \nabla_x P - B'  \nabla u \nabla_q F))
\end{gathered}
\end{equation}
which gives
\begin{equation}\label{ProofLemmaCrit2Eq5}
\begin{gathered}
\sum_{i,j} d_{ij} P_{x_i x_j} + ( \frac{\lambda B'(u)}{A'(| \nabla u |^2)| \nabla u |^2} \nabla u + \nabla_q F) \nabla_x P \geq J
\end{gathered}
\end{equation}
and we conclude.
\end{proof}
$ \\ $

\begin{remark}\label{RmkSmoothnessinLemmas}

If $ F \in C^1 $, the conclusions of Lemmas \ref{LemmaCriterionPfunctionsFullyNonlinearEq} and \ref{LemmaCriterion2PfunctionsFullyNonlinearEq} remain valid for $ u \in C^2 $, with the inequalities written in the viscosity sense. Thus, $ C^2 $ smoothness of the solution $ u $ is enough for Lemmas \ref{LemmaCriterionPfunctionsFullyNonlinearEq} and \ref{LemmaCriterion2PfunctionsFullyNonlinearEq} and the term $ \sum_{i,j} d_{ij} P_{x_i x_j} $ must be a priori interpreted in the sense of distributions and then, combining the third equation in \eqref{ProofLemmaGeneralCritFulNonEq1} with the second equation in \eqref{ProofLemmaGeneralCritFulNonEq2}, we have
\begin{equation}\label{RmkSmoothnesssLemEq1}
\begin{gathered}
\sum_{i,j} d_{ij} P_{x_i x_j} = 2 \sum_{i,j,k} d_{ij} u_{x_k x_j} u_{x_k x_i} - 2 \nabla_x F \nabla u - 2 F_u | \nabla u|^2 - \nabla_q F \nabla_x P \\ + g'(u) \nabla_q F \nabla u
+ g''(u) \sum_{i,j} d_{ij} u_{x_i} u_{x_j} + g'(u) \sum_{i,j} d_{ij} u_{x_i x_j}
\end{gathered}
\end{equation}
and thus $ \sum_{i,j} d_{ij} P_{x_i x_j} \in C( \Omega) $. Similarly, in Lemma  \ref{LemmaCriterion2PfunctionsFullyNonlinearEq}, we combine the first equation in \eqref{ProofLemmaCrit2Eq1} with \eqref{ProofLemmaCrit2Eq3}.
\end{remark}

$ \\ $

A direct consequence of Lemma \ref{LemmaCriterion2PfunctionsFullyNonlinearEq} is the following. 

\begin{corollary}\label{CorolP-functionofDeltau=f(u)}
Let $ u : \Omega \subset \mathbb{R}^n \rightarrow \mathbb{R} $ be a smooth solution of
\begin{equation}\label{Deltau=f(u)}
\Delta u = f(u)
\end{equation}
and let $ P(s,t) = A(t) +B(s) $ such that $ A' >0 $ for $ t>0 \:,\: A'' \geq 0 $ and assume that
\begin{equation}\label{CorolP-functionHypothesis2}
t^2 B''(s) + B'(s) f(s) + \frac{(B'(s))^2}{2A'(t^2)} +2A'(t^2) t^2 f'(s) \geq 0
\end{equation}

Then $ P=P(u, | \nabla u|^2) $ is a $ P -$function of \eqref{Deltau=f(u)}.

\end{corollary}

$ \\ $

\subsection{Examples of $ P -$functions}
$ \\ $
\textbf{(1)} The well known $ P- $function of \eqref{Deltau=f(u)} is
\begin{equation}\label{WellKnownP-functionofDeltau=f(u)}
\begin{gathered}
P(u,| \nabla u |^2) = \frac{| \nabla u |^2}{2} - F(u) \\
\textrm{where} \;\: F'(u) = f(u)
\end{gathered}
\end{equation}
(see \cite{Modica} or Chapter 5 in \cite{Sperb}).

It is easy to see that \eqref{WellKnownP-functionofDeltau=f(u)} satisfies \eqref{CorolP-functionHypothesis2} in Corollary \ref{CorolP-functionofDeltau=f(u)}.

$ \\ $
\textbf{(2)} Another general example of $ P- $function of \eqref{Deltau=f(u)} is
\begin{equation}\label{P-functionofDeltau=f(u)}
\begin{gathered}
P(u,| \nabla u |^2) = \frac{| \nabla u |^4}{2} + 2 \int_0^u (\int_0^y \sqrt{f(z)f'(z)} dz)^2dy \;\;\;\;,\;\; \textrm{if} \;\: f(t)\geq 0 \: \textrm{and} \; f'(t) \geq 0 \;,\: \forall \: t \in \mathbb{R} \\
P(u,| \nabla u |^2) = \frac{| \nabla u |^4}{2} - 2 \int_0^u (\int_0^y \sqrt{-f(z)f'(z)} dz)^2dy \;\;\;\;,\;\; \textrm{if} \;\: f(t) \leq 0 \;\: \textrm{and} \;\: f'(t) \geq 0 , \forall \: t \in \mathbb{R}
\end{gathered}
\end{equation}
and satisfies condition \eqref{CorolP-functionHypothesis2} of Corollary \ref{CorolP-functionofDeltau=f(u)}.

Note that the above example is not in the form  $ P= g(u) | \nabla u |^2 +h(u) $ that we see in \cite{Sperb} as general form for $ P $ related to equation \eqref{Deltau=f(u)}.

$ \\ $
\textbf{(3)} Let $ u $ be a solution of
\begin{equation}\label{PfunctionForF(|gradu|^2-cu,nabla^2u)=0}
F(|\nabla u|^2-cu, \nabla^2 u)=0
\end{equation}
where $ F $ satisfies the ellipticity condition \eqref{EllipticityConditionForFullyNonlinearElEq} and assume $ \sum_{i,j}F_{a_{ij}} u_{x_i x_j} \leq \dfrac{\lambda c}{2} \;,\: c>0. $ 

Then $ P = P(u, | \nabla u |^2) = | \nabla u |^2 - cu $ is a $ P -$function of \eqref{PfunctionForF(|gradu|^2-cu,nabla^2u)=0} since $ P $ satisfies condition \eqref{Lemma2QuantityJforPfunctions} of Lemma \ref{LemmaCriterion2PfunctionsFullyNonlinearEq}.

$ \\ $
\textbf{(4)} The following example is in \cite{PP2} (see Theorem 1). 

Let $ u $ be a solution of
\begin{equation}\label{CGSmoreGeneralEq}
div( \Phi'(| \nabla u |^2) \nabla u) = \rho ( | \nabla u |^2 ) F'(u)
\end{equation}
with $ \Phi'(t), \rho(t) > 0 $ and $ \Phi'(t) +2t \Phi''(t) >0 \;,\: \forall \: t \geq 0 $.

Consider the function
\begin{equation}\label{PfunctionForCGSmoreGeneralEq}
P(s,t) = \int_0^t \frac{\Phi'(y) +2y \Phi''(y)}{\rho(y)} dy - 2F(s)
\end{equation}
Then $ P=P(u, | \nabla u |^2) $ is a $ P -$function of \eqref{CGSmoreGeneralEq}.

Note that for $ \rho \equiv 1 $, we have the one studied in \cite{CGS}.
$ \\ $

\section{Gradient Bounds for entire solutions of Fully Nonlinear equations}

In this section we will see that utilizing the techniques of \cite{CGS,FV2}, we can obtain gradient bounds for solutions of equations of the form \eqref{FullyNonlinearEllipticEq}. To be more precise, for any explicit example of $ P- $function, we obtain a particular gradient bound.

Some of the regularity assumptions in this work can be relaxed for some classes of equations. 
In the study of Quasi-linear equations for example, one can assume that $ u \in W^{1,p}_{loc}(\mathbb{R}^n) \cap L^{\infty}(\mathbb{R}^n) $, as in assumption (i) in Theorem 1.6 in \cite{CGS} and utilize regularity results in \cite{T} afterwords. For fully nonlinear elliptic equations, we can relax the regularity of solutions and then utilize regularity results from \cite{CaffarelliCabre,Trudinger}. However, our main goal is not the optimal regularity assumptions since we state the results in an abstract form. Therefore, we will assume that the solutions are smooth and satisfy an analog of assumption (ii) in Theorem 1.6 in \cite{CGS}. Note that smoothness of solutions is often assumed for fully nonlinear equations (see for example \cite{Trudinger}), that is sufficient to apply Lemmas \ref{LemmaCriterionPfunctionsFullyNonlinearEq} and \ref{LemmaCriterion2PfunctionsFullyNonlinearEq}. One application of interest is a pointwise gradient bound for entire solutions of Pucci's equation that we establish in subsection 3.1. This gradient bound generalizes the Modica inequality in the case of fully nonlinear elliptic equations.

$ \\ $
\textbf{Assumptions}.
\begin{equation}\label{AssumptionOnSolutionsSecondOrder}
\begin{gathered}
u \in C^2_{loc}( \lbrace \nabla u \neq 0 \rbrace )\cap L^{\infty}(\mathbb{R}^n) \;,\: \nabla u \in C^\alpha (\mathbb{R}^n ; \mathbb{R}^n) \;\: \textrm{for some} \;\: \alpha \in (0,1) \\ \textrm{and} \;\: \exists \;\: C >0 \;\: \textrm{depending on} \;\: || u||_{L^{\infty}(\mathbb{R}^n)} \;\: \textrm{such that} \;\: | \nabla u(x)| \leq C\:,\; \forall \;\: x \in \mathbb{R}^n
\end{gathered}
\end{equation}
$ \\ $

The next theorem provides an a priori pointwise estimate for solutions of \eqref{FullyNonlinearEllipticEq}. The theorem below holds for any $ P- $function that satisfies $ P(u,0) \leq 0 $. When $ P $ is of the form $ P=P(u, \nabla u) $ we use the notation $ P(u,0) $ instead of $ P(u,0,...,0) $ and also we sometimes write $ P=P(u;x) $ for simplicity. $ \\ $

\begin{theorem}\label{ThmGradientBoundforDeltau=F}
Let $ u $ be an entire solution of
\begin{equation}\label{ThmGradientBoundNonlinearEq}
F(u, \nabla u, \nabla^2 u) =0
\end{equation}
that satisfy assumption \eqref{AssumptionOnSolutionsSecondOrder}, where $ F $ is continuous degenerate elliptic.
Assume that every bounded entire viscosity solution $ u \in C^{1,\alpha} (\mathbb{R}^n) $ of \eqref{ThmGradientBoundNonlinearEq} belongs in $ \;\: C^2_{loc}( \lbrace \nabla u \neq 0 \rbrace ) $.
If $ P = P(u, \nabla u ) $ is a $ P- $function of \eqref{ThmGradientBoundNonlinearEq}, with $ \mu = \mu(| \nabla u |) $, $ \mu(t)>0 \;,\: \forall \: t>0 $, such that $ P(s,0) \leq 0 $, for $ s \in [ \inf_{\mathbb{R}^n} u , \sup_{\mathbb{R}^n} u  ] $,

Then
\begin{equation}\label{GradientBoundforDeltau=FEq}
P(u(x), \nabla u(x)) \leq 0 \;\;,\; \forall \; x \in \mathbb{R}^n .
\end{equation}
\end{theorem}
$ \\ $
\begin{proof}
Let $ u $ be a solution of \eqref{ThmGradientBoundNonlinearEq} that satisfies assumption \eqref{AssumptionOnSolutionsSecondOrder} and consider the family of all translations of $ u $,
\begin{equation}\label{GeneralGrBoundEq1}
\mathscr{F} = \lbrace v \: : \mathbb{R}^n \rightarrow \mathbb{R} \: | \; \exists \; z \in \mathbb{R}^n \:\: \textrm{such that} \;\: v(x) = u(x+z) \;\: \forall \; x \in \mathbb{R}^n \rbrace
\end{equation}
$ \mathscr{F} $ is non empty since $ u \in \mathscr{F} $.

Let $ P $ be a $ P- $function of \eqref{ThmGradientBoundNonlinearEq}, with $ \mu = \mu(| \nabla u |) $, $ \mu(t)>0 \;,\: \forall t>0 $, such that $ P(u,0) \leq 0 $. For simplicity, we denote $ P=P(u;x) $ instead of $ P=P(u(x), \nabla u(x) ) $. 

Consider now
\begin{equation}\label{GeneralGrBoundEq2}
P_0 = \sup \lbrace P(v;x) \: | \: v \in \mathscr{F} \;,\;\: x \in \mathbb{R}^n \rbrace
\end{equation}

We will prove that
\begin{align}\label{GeneralGrBoundEq2'}
P_0 \leq 0
\end{align}
and from this we conclude.

We argue by contradiction, so we suppose that
\begin{equation}\label{GeneralGrBoundEq2''}
P_0 >0
\end{equation}
Then, by \eqref{GeneralGrBoundEq2} there exist $ (w_k)_{k \in \mathbb{N}} $ in $ \mathscr{F} $ and $ (x_k)_{k \in \mathbb{N}} $ in $ \mathbb{R}^n $ such that
\begin{equation}\label{GeneralGrBoundEq3}
\lim_{k \rightarrow + \infty} P( w_k; x_k) = P_0
\end{equation}

Let $ v_k(x)= w_k(x+x_k) $. Also, by definition we have that $ v_k \in \mathscr{F} $ and $ P(v_k;0) = P(w_k;x_k) $, so that \eqref{GeneralGrBoundEq3} can be rewritten as
\begin{equation}\label{GeneralGrBoundEq4}
\lim_{k \rightarrow + \infty} P( v_k; 0) = P_0
\end{equation}
Since $ v_k \in C^{1,\alpha}_{loc}(\mathbb{R}^n) $, by the Ascoli-Arzela theorem together with a diagonal argument, we can extract from $ (v_k)_{k \in \mathbb{N}} $ a subsequence, denoted by $ (v_k^{(k)})_{k\in \mathbb{N}} $ that converges with its first-order derivatives, uniformly on compact subsets of $ \mathbb{R}^n $. Denote by $ \tilde{v} $ the limit function. 

By the assumption \eqref{AssumptionOnSolutionsSecondOrder} we have that
\begin{equation}\label{GeneralGrBoundEqRelativeCompactness}
\overline{\mathscr{F}} \;\: \textrm{is relatively compact in} \;\: C^{1,\beta}_{loc}(\mathbb{R}^n) \;\:,\; \textrm{for} \;\: 0 < \beta < \alpha
\end{equation}
Thus $ \tilde{v} \in \overline{\mathscr{F}} $, $$ P_0 = \sup_{x \in \mathbb{R}^n , \: v \in \mathscr{F}} P(v;x) = \sup_{x \in \mathbb{R}^n , \: v \in \overline{\mathscr{F}}} P(v;x) $$ and 
\begin{equation}\label{GeneralGrBoundEq4'}
P(\tilde{v};0) = \lim_{k \rightarrow + \infty} P(v_k^{(k)};0 ) = P_0
\end{equation}
by \eqref{GeneralGrBoundEq4}.

Consider now the set
\begin{equation}\label{GeneralGrBoundEq5}
U = \lbrace x \in \mathbb{R}^n \: | \: P(\tilde{v};x) =P_0 \rbrace
\end{equation}
from the continuity of $ P $ on $ \mathbb{R}^n \;,\: U $ is closed and non empty since $ 0 \in U $. We will prove that $ U $ is also open. Let $ x_0 \in U $, we observe that $ |\nabla \tilde{v}(x_0) | \neq 0 $, otherwise we would have
\begin{align*}
P_0 = P(\tilde{v};x_0) = P(\tilde{v}(x_0) , \nabla \tilde{v}(x_0)  ) = P(\tilde{v}(x_0),0) \leq 0
\end{align*}
against the assumption that $ P_0>0 $. 

By continuity, there exists $ \delta >0 $ such that
\begin{equation}\label{GeneralGrBoundEq6}
\inf_{\overline{B}_\delta(x_0)} | \nabla \tilde{v} | >0 \;\; \textrm{and thus} \;\: \inf_{\overline{B}_\delta(x_0)} \mu( | \nabla \tilde{v} |) >0
\end{equation}
and by Theorem \ref{AbstractConsequenceofPfunction1} we conclude that 
\begin{equation}\label{GeneralGrBoundEq6'}
P(\tilde{v};x) \equiv P_0 \;\;\; \textrm{in} \;\: B_{\delta}(x_0) 
\end{equation}
So $ U $ is open and it follows that $ U = \mathbb{R}^n $ by connectedness.

On the other hand, since $ \tilde{v} $ is bounded it holds that there exists a sequence $ (y_l)_{l \in \mathbb{N}} $ in $ \mathbb{R}^n $ such that 
\begin{equation}\label{GeneralGrBoundEq6''}
\lim_{l \rightarrow + \infty} \nabla \tilde{v}(y_l) =0
\end{equation}

By the boundedness of $ \tilde{v} $ we also have $ \tilde{v}(y_l) = \tilde{v}_l \rightarrow v_0 $ up to a subsequence that we still denote as $ y_l $, and so we obtain
\begin{align*}
0< P_0 = \lim_{l \rightarrow \infty} P(\tilde{v}(y_l), \nabla \tilde{v}(y_l) ) = P(v_0,0)
\end{align*}
which contradicts the assumption $ P(s,0) \leq 0 $. Therefore $ P_0 \leq 0 $ and we conclude.
\end{proof}
$ \\ $

\begin{remark}\label{RmkConditionP(u,0)<=0canberemoved}
\textbf{(1)} The regularity assumption \eqref{AssumptionOnSolutionsSecondOrder} is satisfied for bounded weak solutions to Quasi-linear equations studied in \cite{CGS} with smooth potential and also for bounded viscosity solutions to Pucci equations that is the main interest in this work.
$ \\ $
\textbf{(2)} Note also that the condition $ P(u,0) \leq 0 $ can be removed. Consider for example $ P=P(u,\nabla u) $ be a $ P- $function of \eqref{ThmGradientBoundNonlinearEq} such that the condition $ P(u,0) \leq 0 $ is not satisfied and set $ \tilde{P}(u, \nabla u) = P ( u ,\nabla u) - \sup_{\mathbb{R}^n} P(u, 0) $. If $ \sup_{\mathbb{R}^n} P(u, 0) \geq 0 $, then $ \tilde{P} $ is also a $ P- $function of \eqref{ThmGradientBoundNonlinearEq} and satisfies $ \tilde{P}(u,0) \leq 0 $. The only difference is that the gradient bound in this case takes the form
$ P(u(x), \nabla u(x)) \leq \sup_{\mathbb{R}^n} P(u(x),0) . \\ $
\textbf{(3)} The $ C^2 $ regularity of solutions is needed in the set of points where the gradient of $ u $ does not vanish.
\end{remark}

A direct consequence is a general gradient bound for fully nonlinear equations.
\begin{corollary}\label{CorGradientBoundForFullyNonlinear}
Under the assumption of Theorem \ref{ThmGradientBoundforDeltau=F}, assume $ P=P(u, \nabla u) = A(| \nabla u |^2) + B(u) $ be a $ P- $function from Lemma \ref{LemmaCriterion2PfunctionsFullyNonlinearEq} such that $ B(u) \leq 0 $.

Then
\begin{equation}\label{CorGradientBoundFullyNonlinearEq}
\begin{gathered}
| \nabla u |^2 \leq \Psi (u) \\ \textrm{where} \;\: \Psi (s) = A^{-1}(-B(s))
\end{gathered}
\end{equation}
\end{corollary}


$ \\ $

\subsection{A Pointwise Gradient Bound for entire solutions of Pucci's equations}

$ \\ $
We denote as $ \mathscr{M}^- $ and $ \mathscr{M}^+ $ as the Pucci's extremal operators defined in \eqref{Pucci'sOperators}. The first application of Theorem \ref{ThmGradientBoundforDeltau=F} is the following. The above results that are stated for $ \mathscr{M}^- $, hold for $ \mathscr{M}^+ $ as well. $ \\ $

\begin{theorem}\label{ThmGeneralGradientBoundForPucciEq}
Let $ u : \mathbb{R}^n \rightarrow \mathbb{R} $ be a $ C^3 $ bounded solution of
\begin{equation}\label{ThmPucciEq}
\mathscr{M}^- ( \nabla^2 u) = F'(u) 
\end{equation}
and assume that $ F (t) \geq 0 \;\: \textrm{and} \;\: F'' (t) \leq 0 \;,\; \forall t \in [ \inf_{\mathbb{R}^n} u , \sup_{\mathbb{R}^n} u ] .$

Then
\begin{equation}\label{ThmGenGradientBoundStatement}
\frac{1}{2} | \nabla u |^2 \leq \frac{F(u)}{\lambda}
\end{equation}
\end{theorem}
$ \\ $
\begin{proof}
Since $ \mathscr{M}^- $ is concave, by regularity theory in \cite{CaffarelliCabre} (see also \cite{T}) and the boundedness of $ u $, we have that $ u $ satisfies \eqref{AssumptionOnSolutionsSecondOrder}.

First we assume for simplicity that $ \nabla^2 u $ has no zero eigenvalues and it's eigenvalues are distinct.
We set in Lemma \ref{LemmaCriterion2PfunctionsFullyNonlinearEq}, $ A(t) = \frac{t}{2} \;,\; B(s) = - \frac{F(s)}{\lambda} $ and we calculate
\begin{equation}\label{ThmGenGradientBoundProofJquantity}
\begin{gathered}
J = B''(u) \sum_{i,j} d_{ij} u_{x_i} u_{x_j} + F''(u) | \nabla u|^2 + B' \sum_{i,j} d_{ij} u_{x_i x_j} + \lambda (B'(u))^2 \\ 
= \frac{F''(u)}{\lambda} ( \lambda | \nabla u|^2 - \sum_{i,j} d_{ij} u_{x_i} u_{x_j} ) + \frac{F'(u)}{\lambda} ( F'(u) - \sum_{i,j} d_{ij} u_{x_i x_j})
\end{gathered}
\end{equation}
where $ d_{ij} = \dfrac{\partial \mathscr{M}^-(\nabla^2 u)}{\partial u_{x_i x_j}} $.
By ellipticity and by \eqref{ThmPucciEq} it holds $$ \dfrac{F''(u)}{\lambda} ( \lambda | \nabla u|^2 - \sum_{i,j} d_{ij} u_{x_i} u_{x_j} ) \geq 0 . $$ 

Thus
\begin{align*}
J \geq \frac{(F'(u))^2}{\lambda}  - \frac{F'(u)}{\lambda} \sum_{i,j} \frac{\partial \mathcal{M}(\nabla^2 u)}{\partial u_{x_i x_j}}u_{x_i x_j}
\end{align*}
Also we have the following,
\begin{equation}\label{partialderofeigenvalues}
\frac{\partial e_k}{\partial u_{x_i x_j}} =
 v_i^k v_j^k \;\;,\;\: i,j=1,...,n,
\end{equation}
where $ v^k $ is the unit length eigenvector of $ e_k $ (this identity holds since $ e_k $ are distinct).

Identity \eqref{partialderofeigenvalues} is an algebraic fact and can be proved as follows. Let $ A $ be an $ n \times n $ symmetric matrix with eigenvalues $ e_k $ and respective unit length eigenvectors $ v^k $,
\begin{align*}
A \cdot v^k = e_k v^k \;\;\;\;\;\;\;\;\;\;\;\;\;\;\;\;\;\;\;\; \\ \Rightarrow (d A) v^k + A (d v^k) = de_k v^k + e_k dv^k
\end{align*}
where $ d $ is a differential operator, say for example the partial derivative with respect to the element $ a_{ij} $ of $ A $. Since $ v^k $ is unit length, $ v^k dv^k = \frac{1}{2} d(| v^k |^2) = 0 $, so
\begin{align*}
(v^k)^T (d A ) v^k = de_k
\end{align*}
since $( v^k)^T A = e_k (v^k)^T $.
$ \\ $

Thus,
\begin{equation}\label{sumpartialM-calc1}
\begin{gathered}
\sum_{i,j} \frac{\partial \mathscr{M}^-(\nabla^2 u)}{\partial u_{x_i x_j}} u_{x_i x_j} = \lambda \sum_{e_k >0} \sum_{i,j} \frac{\partial e_k}{\partial u_{x_i x_j}} u_{x_i x_j} +\Lambda \sum_{e_k < 0} \sum_{i,j} \frac{\partial e_k}{\partial u_{x_i x_j}} u_{x_i x_j} \\
= \lambda \sum_{e_k >0} \sum_{i,j=1}^n v_i^k v_j^k u_{x_i x_j} +
\Lambda \sum_{e_k <0} \sum_{i,j=1}^n v_i^k v_j^k u_{x_i x_j}
\end{gathered}
\end{equation}
by \eqref{partialderofeigenvalues}.

Also, $ (v^k)^T( \nabla^2 u) v^k = e_k $, since $ (v^k)^T v^k =1 $, so
\begin{equation}\label{sumpartialM-calc2}
\begin{gathered}
\sum_{i,j} \frac{\partial \mathscr{M}^-(\nabla^2 u)}{\partial u_{x_i x_j}} u_{x_i x_j} = \lambda \sum_{e_k >0} e_k + \Lambda \sum_{e_k < 0} e_k = F'(u)
\end{gathered}
\end{equation}
by \eqref{ThmPucciEq} and \eqref{partialderofeigenvalues}.

Therefore,
\begin{align*}
J \geq \frac{(F'(u))^2}{\lambda}  - \frac{F'(u)}{\lambda} \sum_{i,j} \frac{\partial \mathscr{M}^-(\nabla^2 u)}{\partial u_{x_i x_j}}u_{x_i x_j}  = 0
\end{align*}
thus, by Lemma \ref{LemmaCriterion2PfunctionsFullyNonlinearEq}, we have that $ P(u, | \nabla u |^2) = \frac{1}{2} | \nabla u |^2 - \frac{F(u)}{\lambda} $ is a $ P- $function of \eqref{ThmPucciEq}. 

In addition, $ P (u,0) = - \frac{F(u)}{\lambda} \leq 0 $, therefore by Theorem \ref{ThmGradientBoundforDeltau=F} we conclude that
\begin{align*}
\frac{1}{2} | \nabla u |^2 \leq \frac{F(u)}{\lambda} \;\;\;,\;\; \forall x \in \mathbb{R}^n.
\end{align*}

In the case where do not impose any assumption on the eigenvalues of the hessian matrix of $ u $, we argue as follows.

Let $ x_0 \in \Omega:= \mathbb{R}^n \cap \lbrace \nabla u \neq 0 \rbrace $ and consider the set $ \mathcal{A}_{\lambda, \Lambda} $ of symmetric $ n \times n$ matrices
such that their quadratic form are bounded from below by $ \lambda $ and bounded from above by $ \Lambda $. Then $ \mathscr{M}^- $ can be written as
\begin{equation}\label{WeakPucciEq}
\mathscr{M}^- ( \nabla^2 u(x_0)) = \textrm{inf}_{A \in \mathcal{A}_{\lambda, \Lambda}} \textrm{tr} ( A \cdot \nabla^2 u(x_0)) =F'(u(x_0))
\end{equation}
and we can choose matrix $ A_0 \in \mathcal{A}_{\lambda, \Lambda} $ so that the infimum is attained (see for example Remark 5.36 in \cite{ACM}). Now we want to show the analog of \eqref{sumpartialM-calc2} without differentiating directly the operator $ \mathscr{M}^- $. 

Set $$ G(x) = tr(A_0 \nabla^2 u(x)) - F'(u(x)) \:, $$ then $ h(x) \geq 0 $ and $ h(x_0) =0 $ since $ \mathscr{M}^- (\nabla^2 u(x)) \leq A_0 \nabla^2 u(x) \;,\: \forall x \in \mathbb{R}^n $ and $ G \in C^1 $ since $ u \in C^3 $ and $ F \in C^2 $. 

Therefore $ \nabla G(x_0) = 0 $ and
\begin{align*}
\textrm{tr} ( A_0 \cdot \nabla^2 u_{x_k} (x_0)) = F''(u)u_{x_k}
\end{align*}
and calculating as in Lemma \ref{LemmaCriterion2PfunctionsFullyNonlinearEq}, utilizing ellipticity and the assumptions of $ F $ we get
\begin{align*}
\textrm{tr} ( A_0 \cdot \nabla^2 P) - \frac{2F'(u)}{| \nabla u|^2} \nabla u \nabla P \geq \lambda \frac{| \nabla P|^2 }{| \nabla u|^2} \geq 0
\end{align*}
Since $ x_0 \in \Omega $ is arbitrary, for every $ x \in \Omega $ we may choose a minimizing matrix $ A(x) \in \mathcal A_{\lambda,\Lambda} $. The preceding calculation then yields that
%
$ P = \frac{1}{2} | \nabla u |^2 - \frac{F(u)}{\lambda} $ is a $ P- $function and we conclude.
\end{proof}
$ \\ $

\begin{remark} The gradient bound \eqref{ThmGenGradientBoundStatement} also holds for the operator $ \mathscr{M}^+ $ with similar calculations. The regularity theory from \cite{CaffarelliCabre} can be applied in this case since $ \mathscr{M}^+ $ is convex.
\end{remark}

$ \\ $

Now we provide another pointwise gradient bound for Pucci's equations that generalizes the Modica inequality.

\begin{theorem}\label{ThmModicaGradientBoundForPucciEq}
Let $ u : \mathbb{R}^n \rightarrow \mathbb{R} $ be a $ C^3 $ bounded solution of
\begin{equation}\label{ThmPucciEqwithACpotetnial}
\mathscr{M}^- ( \nabla^2 u) = u^3 -u .
\end{equation}

Then
\begin{equation}\label{ThmModicaGradientBoundStatement}
\frac{1}{2} | \nabla u |^2 \leq \frac{(1-u^2)^2}{4\lambda} - g(u)
\end{equation}
where $ g : \mathbb{R} \rightarrow \mathbb{R} $ is defined as
\begin{equation}\label{ThmModicaGradientBoundfunctiong}
g(u) = \begin{cases} \frac{\Lambda - \lambda}{\lambda \Lambda} \left[ \frac{(1-u^2)^2}{4} + \frac{2 (|u|-1)}{3 \sqrt{3}}  \right] \;\;\;,\;\; 1> | u| > \frac{1}{\sqrt{3}} \\
\frac{\Lambda - \lambda}{\lambda \Lambda} \left( \frac{1}{3} - \frac{2}{3 \sqrt{3}} \right) \;\;\;\;\;\;\;,\;\; | u| \leq \frac{1}{\sqrt{3}}
\end{cases}
\end{equation}
\end{theorem}
$ \\ $

\begin{proof}
We argue as in the proof of Theorem \ref{ThmGeneralGradientBoundForPucciEq} with the following modifications. We have $ F(u) = \dfrac{(1-u^2)^2}{4} $ and take $ A(t) = t/2 $ and $ B(s) = - \dfrac{F(u)}{\lambda} + g(u) $ where $ g $ is defined in \eqref{ThmModicaGradientBoundfunctiong}. First consider the case where $ \nabla^2 u $ has no zero eigenvalues and it's eigenvalues are distinct. So the quantity $ J $ in Lemma \ref{LemmaCriterion2PfunctionsFullyNonlinearEq} in \eqref{ThmGenGradientBoundProofJquantity} becomes
\begin{equation}\label{ThmModicaForPucciEqProofJquantity}
J= B''(u) \sum_{i,j} d_{ij} u_{x_i} u_{x_j} + F''(u) | \nabla u |^2 + B'(u) \sum_{i,j} d_{ij} u_{x_i x_j} + \lambda (B'(u))^2
\end{equation}
Since $ g''(u) = \dfrac{\Lambda - \lambda}{\lambda \Lambda}F''(u) $ if $ | u| > \frac{1}{\sqrt{3}} $ and $ g(u) =0 $ if $ | u| \leq  \frac{1}{\sqrt{3}} $, direct calculations give that
$$ B''(u) \sum_{i,j} d_{ij} u_{x_i} u_{x_j} + F''(u) | \nabla u |^2 = \begin{cases} \frac{F''(u)}{\Lambda} ( \Lambda |\nabla u|^2 - \sum_{i,j} d_{ij} u_{x_i} u_{x_j} ) \;\:,\; \textrm{if} \;\:  | u| > \frac{1}{\sqrt{3}} \\ \frac{F''(u)}{\lambda} ( \lambda |\nabla u|^2 - \sum_{i,j} d_{ij} u_{x_i} u_{x_j} ) \;\:,\; \textrm{if} \;\:  | u| \leq \frac{1}{\sqrt{3}}
\end{cases} $$
recall that $ d_{ij} = \dfrac{\partial \mathscr{M}^-}{\partial u_{x_i x_j}} $ and in both cases $  B''(u) \sum_{i,j} d_{ij} u_{x_i} u_{x_j} + F''(u) | \nabla u |^2 \geq 0 $ by ellipticity of $ \mathscr{M}^- $ and the particular sign of $ F''(u) = 3u^2 -1 $.

Thus
\begin{align*}
J \geq B'(u) (F'(u) + \lambda B'(u)) = \lambda (g'(u))^2 - F'(u)g'(u)
\end{align*}
by \eqref{sumpartialM-calc2}. If $ | u| \leq \frac{1}{\sqrt{3}} $ then $ g' =0 $ and $ J \geq 0 $.

Now, since $ u $ is bounded it holds that $ | u | \leq 1 $. 

Thus, if $ 1 \geq  u > \frac{1}{\sqrt{3}} \;,\: F'(u) =u^3 -u \leq 0 $ and $ g'(u) = \frac{\Lambda - \lambda}{\lambda \Lambda} (u^3 - u + \frac{2}{3 \sqrt{3}}) \geq 0 $. 

If $ -1 \leq u < - \frac{1}{\sqrt{3}} \;,\: F'(u)=u^3-u \geq 0 $ and $ g'(u) = \frac{\Lambda - \lambda}{\lambda \Lambda} (u^3 - u - \frac{2}{3 \sqrt{3}}) \leq 0 .\\ $

Therefore we conclude that $ J \geq 0 . \\ $

In the case where do not impose any assumption on the eigenvalues of the hessian matrix of $ u $, we argue as in the final part of the proof of Theorem \ref{ThmGeneralGradientBoundForPucciEq} above.

\end{proof}
$ \\ $

\begin{remark}\label{RmkModicaforPucciEq}
\textbf{(1)} Note that for $ \lambda = \Lambda=1 $, the operator $ \mathscr{M}^- $ equals the Laplacian and the gradient bound \eqref{ThmModicaGradientBoundStatement} becomes \eqref{ModicaInequality}. $ \\ $
\textbf{(2)} The same gradient bound holds also for the operator $ \mathscr{M}^+ $ and the proof is similar. $ \\ $
\textbf{(3)} The fact that a nonconstant bounded solution $ u $ from Theorem \ref{ThmModicaGradientBoundForPucciEq} satisfies $ |u|<1 $ can be easily obtained by Theorem \ref{ThmLiouvilleType} in the next section.
\end{remark}

$ \\ $

\subsection{Gradient Bounds for entire solutions by the Examples of subsection 2.1}
$ \\ $
Next, we observe that in Example (3), we have the following gradient bound

\begin{corollary}\label{CorGradBoundEx4} Let $ u $ be a non negative entire solution of
\begin{equation}\label{CorGradientBoundEx4}
F(|\nabla u|^2-cu, \nabla^2 u)=0
\end{equation}
that satisfy \eqref{AssumptionOnSolutionsSecondOrder}, where $ F $ satisfies the ellipticity condition \eqref{EllipticityConditionForFullyNonlinearElEq} and assume $ \\ \sum_{i,j}F_{a_{ij}} u_{x_i x_j} \leq \dfrac{\lambda c}{2} $ for some $ c>0 $.

Then
\begin{equation}\label{CorGradBoundEx4Eq}
| \nabla u |^2 \leq c u
\end{equation}
\end{corollary}
\begin{proof}
We have that the function
\begin{align*}
P(s,t) = t-cs
\end{align*}
satisfy the condition \eqref{Lemma2QuantityJforPfunctions} and also, $ P(u,0) = -cu \leq 0 $ since $ u $ is non negative by assumption. Therefore we conclude by the Theorem \ref{ThmGradientBoundforDeltau=F}.
\end{proof}

$ \\ $

The following gradient bound that we derive from Theorem \ref{ThmGradientBoundforDeltau=F}, via Example (4), is a quite more general form of the gradient bound in \cite{CGS}. $ \\ $

\begin{corollary}\label{GradientBoundforCGSmoreGeneralEq}
Let $ u $ be an entire solution of
\begin{equation}\label{CorolCGSmoreGeneralEq}
div( \Phi'(| \nabla u |^2) \nabla u) = \rho ( | \nabla u |^2 ) F'(u) \;\;,\; F \geq 0
\end{equation}
that satisfy assumption \eqref{AssumptionOnSolutionsSecondOrder}, with $ \Phi'(t), \rho(t) > 0 $ and $ \Phi'(t) +2t \Phi''(t) >0 \;,\: \forall \: t \geq 0 $.

Then
\begin{equation}\label{GradientBoundCGSmoreGeneralEq}
\begin{gathered}
| \nabla u |^2 \leq \Psi (u) \;\;, \; \textrm{where} \;\: \Psi(u) = Q^{-1}(2 F(u)) \\
\textrm{and} \;\: Q(t) = \int_0^t \frac{\Phi'(y) +2y \Phi''(y)}{\rho(y)} dy
\end{gathered}
\end{equation}
\end{corollary}
$ \\ $
\begin{proof}
By Theorem 1 in \cite{PP2}, we have that $ P(u, | \nabla u |^2) = Q(| \nabla u |^2) -2 F(u) $ is a $ P- $function of \eqref{CorolCGSmoreGeneralEq} with $ \mu(t) >0 \;,\: \forall \: t \geq 0 $ and satisfies $ P(u,0) \leq 0 $ since $ F \geq 0 $. Thus we apply Theorem \ref{ThmGradientBoundforDeltau=F} and we conclude.
\end{proof}

$ \\ $

\section{Rigidity results and properties of entire solutions of fully nonlinear equations}

In this section we will see that if an equation admits a $ P- $function of the form $ P= | \nabla u|^2 $, then the solutions that satisfy assumption \eqref{AssumptionOnSolutionsSecondOrder} are constant. As a result we have a Liouville theorem for Pucci's equations and special cases are some of the well-known results of J. Serrin in \cite{Serrin}.

In particular we have

\begin{theorem}\label{theoremRigidityForFullyNonlinear}
Let $ u $ be an entire solution of
\begin{equation}\label{thmRigidityFullyNonlinearEq}
F(u, \nabla u, \nabla^2 u) =0
\end{equation}
that satisfy assumption \eqref{AssumptionOnSolutionsSecondOrder} where $ F $ satisfies the condition \eqref{EllipticityConditionForFullyNonlinearElEq} and assume $ F_u \leq 0 $.

Then $ u $ is a constant
\end{theorem}
$ \\ $
\begin{proof}
The proof of Theorem \ref{theoremRigidityForFullyNonlinear} is a consequence of Theorem \ref{ThmGradientBoundforDeltau=F} and Lemma \ref{LemmaCriterionPfunctionsFullyNonlinearEq} by considering $ P = | \nabla u|^2 $.
\end{proof}
$ \\ $
\begin{remark}\label{RmkStability} \textbf{(1)} Note for example that Theorem 1 and Theorem 6 in \cite{Serrin} are recovered.
$ \\ $
\textbf{(2)} We note that, in the special case where equation \eqref{thmRigidityFullyNonlinearEq} takes the form
\begin{equation}\label{RmkStabilityEq}
div( \Phi'( | \nabla u |^2 \nabla u) = f'(u)
\end{equation}
then the condition $ F_u \leq 0 $ reads $ f''(u) \geq 0 $ which implies stability of the solutions, i.e. the second variation of the associated energy functional $ J(u) = \int ( \frac{1}{2} \Phi(| \nabla u |^2) + f(u))dx $ is non negative (see also Theorem 4.5 in \cite{CGS}).
\end{remark}
$ \\ $

In addition, we have a Liouville-type result for Pucci's equation as a direct application of Theorem \ref{theoremRigidityForFullyNonlinear}. $ \\ $

\begin{corollary}\label{LiouvilleThmForPucciEq}
Let $ u $ be an entire solution of
\begin{equation}\label{LiouvForPucciEquation}
\mathscr{M}^- ( \nabla^2 u) =F'(u)
\end{equation}
that satisfy \eqref{AssumptionOnSolutionsSecondOrder} and assume that $ F''(u) \geq 0 $.

Then $ u $ is a constant.
\end{corollary}
$ \\ $

Note that in the proof of Corollary \ref{LiouvilleThmForPucciEq}, we can make use of the argument in the last part in the proof of Theorem \ref{ThmGeneralGradientBoundForPucciEq} to avoid differentiating $ \mathscr{M}^- $. Corollary \ref{LiouvilleThmForPucciEq} also holds for $ \mathscr{M}^+ $.
Another consequence of Theorem \ref{ThmGradientBoundforDeltau=F} is the following Liouville-type result
$ \\ $

\begin{theorem}\label{ThmLiouvilleType}
Under the assumptions of Theorem \ref{ThmGradientBoundforDeltau=F} and let $ P $ is a $ P - $function from Lemma \ref{LemmaCriterion2PfunctionsFullyNonlinearEq} that satisfies $ A(0) =0 \;,\: B(u) \leq 0 $. If there exists $ x_0 \in \mathbb{R}^n $ such that $ B (u(x_0)) = B' (u(x_0))  =0 $, then $ u \equiv \textrm{const.} $ in $ \mathbb{R}^n $.
\end{theorem}
\begin{proof}
We argue as in the proof of Theorem 1.8 in \cite{CGS} with slight modifications. For the convenience of the reader we provide the details.

Suppose that $ B (u(x_0))=0 $, let $ u_0=u(x_0) $ and consider the set
\begin{equation}\label{proofThmLiouvilletype1}
V= \lbrace x \in \mathbb{R}^n \: | \: u(x) =u_0 \rbrace
\end{equation}
$ V $ is a closed set and by the assumption, non empty. Let $ x_1 \in V $ and consider the function $ \phi (t) = u(x_1 + t \omega) -u_0 $, where $ |\omega| =1 $ is arbitrarily fixed. We have $ | \phi '(t) | = | \nabla u(x_1 +t \omega) | $. By the gradient bound in Corollary \ref{CorGradientBoundForFullyNonlinear} we have,
\begin{equation}\label{proofThmLiouvilletype2}
| \nabla u |^2 \leq \Psi(u) \;\;,\; \textrm{where} \;\: \Psi(s) = A^{-1}( -B(s))
\end{equation}
Since $ \Psi \in C^2(\mathbb{R}) $ and $ \Psi(u_0)=0 $, we have $ \Psi(u) = O(|u-u_0|^2) $, as $ | u-u_0 | \rightarrow 0 $. So, we conclude from \eqref{proofThmLiouvilletype2} that $ | \phi'(t) | \leq C | \phi(t) | $ for $ t $ small enough. Since $ \phi(0)=0 $, we must have $ \phi \equiv 0 $ on $ [-\delta, \delta] $, for some $ \delta >0 $. Thus $ V $ is open, which gives that $ V =\mathbb{R}^n $.
\end{proof}
$ \\ $

Also, we have a De Giorgi type property for solutions that attain the equality at a point in the gradient bound obtained in Corollary \ref{CorGradientBoundForFullyNonlinear}. $ \\ $

\begin{theorem}\label{ThmDeGiorgitypeforDeltau=F}
Let $ u $ be an entire solution of
\begin{equation}\label{ThmDeGiorgitypeforEquationDeltau=F}
F(u, \nabla u, \nabla^2 u)=0
\end{equation}
under the assumptions of Theorem \ref{ThmLiouvilleType}, assume $ P=P(u,| \nabla u |^2 ) $ is a $ P- $function of \eqref{ThmDeGiorgitypeforEquationDeltau=F} obtained in Lemma \ref{LemmaCriterion2PfunctionsFullyNonlinearEq}.
If there exists $ x_0 \in \mathbb{R}^n $ such that
\begin{equation}\label{P=B-Gammaattainszero}
P(u(x_0), | \nabla  u(x_0) |^2 ) = 0
\end{equation}
then there exists a function $ g : \mathbb{R} \rightarrow \mathbb{R} $ such that
\begin{equation}\label{ThmDeGiorgiforDeltau=Fuonedim}
\begin{gathered}
\textrm{either} \;\;\;\:  u(x) = g(a \cdot x +b) \;,\; a \in \mathbb{R}^n \;\: \textrm{with} \;\: |a|=1 \;,\: b \in \mathbb{R} \:, \\ \textrm{or} \;\;\;\: u \;\: \textrm{is a constant in} \;\: \mathbb{R}^n
\end{gathered}
\end{equation} 
\end{theorem}
$ \\ $
\begin{proof}
By Corollary \ref{CorGradientBoundForFullyNonlinear}, we have that $ P (u, | \nabla u|^2) \leq 0 $. 

We begin by considering the set
\begin{equation}\label{ThmDeGiorgiforDeltau=FA=(P=0)}
\mathscr{A}= \lbrace x \in \mathbb{R}^n \: : \: P (u, | \nabla u |^2) =0 \rbrace
\end{equation}
$ \mathscr{A} $ is closed and by the assumption $ \mathscr{A} \neq \emptyset $. We are going to prove that $ \mathscr{A} $ is open.

Let $ x_1 \in \mathscr{A} $, if $ \nabla u(x_1) =0 $, we obtain that $ P(u(x_1),0)= - B(u(x_1)) =0 $. By Theorem \ref{ThmLiouvilleType}, we conclude that $ u \equiv u(x_1) $ and $ \nabla u \equiv 0 $ and hence $ P \equiv 0 $.

On the other hand, if $ \nabla u(x_1) \neq 0 $, we have $ \inf_{\overline{B}_{\delta_1}(x_1)} | \nabla u | >0 $ for some $ \delta_1 >0 $ 
and by Theorem \ref{AbstractConsequenceofPfunction1} we conclude that $ P(u, | \nabla u |^2) \equiv 0 $ in $ B_{\delta_1} (x_1) $ and therefore $ \mathscr{A} $ is open.

By connectedness, we have that $ \mathscr{A} = \mathbb{R}^n $, that is,
\begin{equation}\label{ThmDeGiorgiforDeltau=F,P=0}
P(u, | \nabla u |^2 ) \equiv 0 \;\;,\; \forall \: x \in \mathbb{R}^n
\end{equation}
and $ P_t = A'(t) >0 $, thus
\begin{equation}\label{ThmDeGiorgiforDeltau=FEikonal}
| \nabla u |^2 = \Psi (u) \;\;\;,\;\; \textrm{in} \;\: \mathbb{R}^n \;\;,\; \textrm{where} \;\: \Psi(u) = A^{-1}(- B (u))
\end{equation}
Now, if there exists $ x_2 \in \mathbb{R}^n $ such that $ \Phi (u(x_2))=0 $, so $ | \nabla u(x_2) |=0 $, again by Theorem \ref{ThmLiouvilleType} we have that $ u \equiv u(x_2) $.

If, on the other hand $ \Psi (u(x)) >0 \;\:,\; \forall \: x \in \mathbb{R}^n $, we set
\begin{equation}\label{ThmDeGiorgiforDeltau=Fv=G(u)Eikonal=1}
\begin{gathered}
v = G(u) \;\;\;,\;\; \textrm{where} \;\: G'(s) = \frac{1}{ \Psi (s)} \\
\textrm{and} \;\: | \nabla v |^2 =1 \;\;\; \textrm{in} \;\: \mathbb{R}^n
\end{gathered}
\end{equation}
Therefore, by the result in \cite{CC}, we have that
\begin{equation}\label{ThmDeGiorgiforDeltau=FEntireSolutionsofEikonalCC}
\begin{gathered}
\textrm{either} \;\: v(x) = a \cdot x +b \;\:,\; a \in \mathbb{R}^n \;\: \textrm{with} \;\: | a |=1 \;\: \textrm{and} \;\:  b \in \mathbb{R} \\
\textrm{or} \;\: v(x) = | x-z_0 | +c \;\:,\; z_0 \in \mathbb{R}^n \;\: \textrm{and} \;\:  c \in \mathbb{R}
\end{gathered}
\end{equation}
The radial solutions are excluded since, in the case of nonconstant solution, if $ y_0 \in \mathbb{R}^n $ is such that $ P(u;y_0) = 0 $ and $ \nabla u( y_0) =0 $, we have $ A(0)=0 $ and $ B(u(y_0))=0 $ and by Theorem \ref{ThmLiouvilleType} $ u $ is a constant. Therefore a nonconstant solution satisfies $ \nabla u \neq 0 $ in $ \mathbb{R}^n $.

So we conclude that
\begin{equation}
\begin{gathered}
\textrm{either} \;\;\;\: 
u(x) = g( a \cdot x + b) \;\:,\; a \in \mathbb{R}^n \;\: \textrm{with} \;\: | a |=1 \;,\;\:  b \in \mathbb{R} \;\: \textrm{where} \;\: g(s) = G^{-1}(s)
\\ \textrm{or} \;\;\;\: \textrm{or} \;\;\;\: u \;\: \textrm{is a constant in} \;\: \mathbb{R}^n
\end{gathered}
\end{equation}
\end{proof}

\begin{remark}\label{RmkConstancyofPfunction} Note that if $ u : \Omega \rightarrow \mathbb{R} $ where $ \Omega $ is an open and connected domain in $ \mathbb{R}^n $ and $ P = P(u,| \nabla u |^2) =A(| \nabla u |^2) + B(u) $ with $ A'>0 $ that attains its maximum at a point then $ u $ will be a solution of the Eikonal equation $ | \nabla u|^2 = \Psi(u) $. If in addition $ u_{x_n}>0 $ and consider $ F_i = \frac{u_{x_i}}{u_{x_n}} $, by Proposition 2.1 in \cite{G}, the function $ F=(F_1,...,F_{n-1}) $ will satisfy the Isobaric Euler equation.
\end{remark}
$ \\ $

\section{A Harnack-type inequality and Local Estimates for the gradient}

We will establish a Harnack inequality and local estimates for the gradient of solutions to to fully nonlinear elliptic equations in a domain $ \Omega \subset \mathbb{R}^n $. By Lemma \ref{LemmaCriterionPfunctionsFullyNonlinearEq} with $ g= 0$ or  by Lemma \ref{LemmaCriterion2PfunctionsFullyNonlinearEq} with $ B=0 $, we can utilize the elliptic inequality \eqref{ProofLemmaGeneralCritFulNonEq3} or \eqref{ProofLemmaCrit2Eq5} respectively and assuming that $ F_{a_{ij}} \in L^{\infty}(\Omega) $, we can apply Local Estimates for subsolutions to general elliptic operators to the $ P- $function and obtain similar local estimates for the gradient of solutions to fully nonlinear equations. $ \\ $

So, we consider solutions of the equation
\begin{equation}\label{HarnackSectionFullyNonlinearEq}
F(x,u,\nabla u, \nabla^2 u) = 0
\end{equation}
where $ F : \Omega \times \mathbb{R} \times \mathbb{R}^n \times \mathbb{R}^{n \times n} \rightarrow \mathbb{R} $ is a continuous function and satisfies the ellipticity condition \eqref{EllipticityConditionForFullyNonlinearElEq}. We denote $ F = F(x,s,q, A) \;\:,\; F_{q_i}= \dfrac{\partial F}{\partial q_i} \;,\; F_{a_{ij}} = \dfrac{\partial F}{\partial a_{ij}} $.

In this section we will assume the bound
\begin{equation}\label{BoundAssumptionHarnackSection}
| F_{a_{ij}} |, | F_{q_i} |, |F_s| \leq M .
\end{equation}
$ \\ $

We first establish a local pointwise estimate for the gradient of solutions.
$ \\ $
\begin{theorem}\label{thmlocalpointwiseestimateofGradientonFullyNonlinear}
Let $ u : \Omega \rightarrow \mathbb{R} $ be a smooth solution of \eqref{HarnackSectionFullyNonlinearEq} and assume $ \nabla_x F \cdot \nabla u \leq 0 $. Then for any $ B_{2R} \subset \Omega $ and $ p >0 $
\begin{equation}\label{thmlocalpointwiseestimateofGradientonFullyNonlinearEstimate}
\sup_{B_R} | \nabla u |^2 \leq C \frac{|| \nabla u  ||^2_{L^{2p}(B_{2R})}}{| B_{2R} |^{1/p}}
\end{equation}
where $ C = C(n,p ,\lambda, \Lambda, MR^2) $ is a positive constant.
\end{theorem}
$ \\ $
\begin{proof} Set $ P=| \nabla u |^2 $, by Lemma \ref{LemmaCriterionPfunctionsFullyNonlinearEq} and the assumption $ \nabla_x F \cdot \nabla u \leq 0 $ we have
\begin{equation}\label{ProofthmlocalpointestofGradonFullyNonlinearEq1}
\sum_{i,j} d_{ij} P_{x_i x_j} + \nabla_q F \nabla_x P + 2 F_u P \geq 0
\end{equation}
where $ d_{ij} = F_{a_{ij}} $.
Thus, by classical local pointwise estimates for subsolutions of elliptic equations (see for example Theorem 9.20 in \cite{GT}) we obtain \eqref{thmlocalpointwiseestimateofGradientonFullyNonlinearEstimate}.
\end{proof}
$ \\ $

Next, we obtain the following Harnack estimate $ \\ $

\begin{theorem}\label{thmHarnackIneqForFullyNonlinearEq}
Let $ u : \Omega \rightarrow \mathbb{R} $ be a smooth solution of \eqref{HarnackSectionFullyNonlinearEq} and assume $ \nabla_x F \cdot \nabla u \geq 0 $ , $  | \nabla^2 u |^2 \in L^n(\Omega) $. 
Then for any $ B_{2R} \subset \Omega $ there holds that for $ p >0 $,
\begin{equation}\label{thmHarnackInFullyNonlinearStatement}
\left( \frac{1}{| B_R |} \int_{B_R} | \nabla u |^{2p} \right)^{\frac{1}{p}} \leq C \left( \inf_{B_R} | \nabla u |^2 + \frac{R}{\lambda} || \nabla^2 u||_{L^{2n}(B_{2R})}^2 \right) ,
\end{equation}
where $ C = C(n,p, \lambda, \Lambda, M R^2) .$
\end{theorem}
$ \\ $
\begin{proof} Set $ P=| \nabla u |^2 $. We argue as in the proof of Lemma \ref{LemmaCriterionPfunctionsFullyNonlinearEq} and by the assumption $ \nabla_x F \cdot \nabla u \geq 0 $ to obtain
\begin{equation}\label{ProofthmHarnackIneqForFullyNonlinearEq1}
\sum_{i,j} d_{ij} P_{x_i x_j} + \nabla_q F \nabla_x P + 2 F_u P \leq 2 \Lambda | \nabla^2 u |^2
\end{equation}
So, by the Harnack inequality for supersolutions of elliptic equations (see Theorem 9.22 in \cite{GT} for instance) we conclude.
\end{proof}
$ \\ $

In the end of this section, we conclude with the following Harnack inequality for the Monge-Ampère equation 
$ \\ $
\begin{corollary}\label{CorollaryHarnackForMongeAmpere} Let $ u : \Omega \rightarrow \mathbb{R} $ be a smooth and convex solution of
\begin{equation}\label{CorHarnMongeAmpereEq}
det( \nabla^2 u) = f(u, \nabla u)
\end{equation}
where $ f > 0 $ and assume $ | \nabla^2 u |^2 \in L^n(\Omega) $ and $ | f_{q_i} |, | adj^T(\nabla^2 u)_{ij} | \leq M . $

Then for any $ B_{2R} \subset \Omega $ we have
\begin{equation}\label{CorHarnMongeAmpereEstimate}
\sup_{B_R} | \nabla u |^2 \leq C \left( \inf_{B_R} | \nabla u |^2 + \frac{R}{\lambda} || \nabla^2 u||_{L^{2n}(B_{2R})}^2  \right)
\end{equation}
where $ C= C(n,\lambda, \Lambda,MR^2) . $
\end{corollary}
\begin{proof}
The proof is a consequence of Theorem \ref{thmlocalpointwiseestimateofGradientonFullyNonlinear} and Theorem \ref{thmHarnackIneqForFullyNonlinearEq} and since $ \frac{\partial F}{\partial a_{ij}}(A) = adj^T(A)_{ij} $, for $ F(A)= det(A) $ by Jacobi's formula.
\end{proof}

$ \\ $

\section{Higher order nonlinear equations}

In this last section, we will provide examples of $ P -$functions for higher order nonlinear equations and their applications. In particular, an analogous version of Theorems \ref{ThmGradientBoundforDeltau=F} and \ref{theoremRigidityForFullyNonlinear}, allow us to obtain properties and pointwise estimates of entire solutions even in this case. Moreover, we establish local pointwise estimates for nonlinear equations of order greater than two, through the mean value properties of the $ P- $functions or with analogous arguments to that of section 5, applied in higher order equations. This method can be applied to many other classes of higher order nonlinear equations. 

We begin by stating the analogous Theorem \ref{ThmGradientBoundforDeltau=F} for equations of general order. $ \\ \\ $
\textbf{Assumption}
\begin{equation}\label{AssumptionHigherOrderPDEs}
\begin{gathered}
u \in C^m_{loc}( \lbrace \nabla u \neq 0 \rbrace ) \cap L^{\infty} (\mathbb{R}^n) \;\;,\; \nabla^{m-1}u \in C^{\alpha}(\mathbb{R}^n) \;\: \textrm{for some} \;\: \alpha \in (0,1) \\
\textrm{and there exists} \;\: C >0 \;\: \textrm{such that} \;\: | \nabla^l u | \leq C \;\:,\; l=1,...,m-1.
\end{gathered}
\end{equation}
$ \\ $

\begin{theorem}\label{ThmBoundForHigherOrderPDEs}
Let $ u $ be an entire solution of
\begin{equation}\label{ThmBoundHigherOrderGeneralEq}
F( u, \nabla u,..., \nabla^m u) =0 
\end{equation}
that satisfies assumption \eqref{AssumptionHigherOrderPDEs} and let $ P = P(u,\nabla u, \Delta u,..., \Delta^l u) =P(u;x) $,  $ \; 2l \leq m-1 \;, $  $ m \geq 3 $, be a $ P- $function of \eqref{ThmBoundHigherOrderGeneralEq}, such that one of the following holds: $ \\ $
(i) $ \mu = \mu (g(\nabla u)) $ for some $ g :\mathbb{R}^{n} \rightarrow \mathbb{R} \;,\: g(z) >0 \;, \forall \: z \neq 0 \;,\: g((0,...,0))=0 \;, \\ \mu(t) >0 \;,\: \forall \: t>0 \;,\: P(u;x) \leq 0 \;, $ when $ \nabla u=(0,...,0) \;,\: $ and $ g(\nabla u)>0 \;,\: \forall \; x \in \mathbb{R}^n , \\ $
(ii) $ \mu = \mu (g(\Delta^k u)) $ for some $ g : \mathbb{R} \rightarrow \mathbb{R} \;,\: g(z) >0 \;, \forall \: z \neq 0 \;,\: g((0,...,0))=0 \;, \\ \mu(t) >0 \;,\: \forall \: t>0 $ and $ P(u;x) \leq 0 \;, $ when $ \Delta^k u=(0,...,0) \;,\; k \geq 1 \;,\: \textrm{and} \;\: 2k \leq m-1 \: . \\ $
Assume that the set of translation of solutions is closed, either in the open set $ \lbrace \nabla u \neq 0 \rbrace $ if (i) holds or in the open set $ \lbrace \Delta^k u \neq 0 \rbrace $ if (ii) holds, with respect to $ C^{m-1}_{loc} ( \mathbb{R}^n ) $ norm.

Then $ P(u,\nabla u, \Delta u,..., \Delta^{m-2}u) \leq 0 \;\: \forall \: x \in \mathbb{R}^n. $

\end{theorem}
\begin{proof}
$ \\ $
If (i) holds, the proof is similar to that of Theorem \ref{ThmGradientBoundforDeltau=F}.
$ \\ $
If (ii) holds, we have the following minor modifications in the proof of Theorem \ref{ThmGradientBoundforDeltau=F}.

It suffices to prove that the set $ U $ defined in \eqref{GeneralGrBoundEq5} is open and conclude the contradiction argument. For simplicity take $ m=3 $ and let $ x_0 \in U $ such that $$ \inf_{\overline{B}_{\delta}(x_0)} g( \Delta u) > 0 \;\: \textrm{and thus} \;\: \inf_{\overline{B}_{\delta}(x_0)} \mu( g( \Delta u)) >0 $$
and by Theorem \ref{AbstractConsequenceofPfunction1} we have
\begin{equation}\label{ProofThmGrBHigherOrdEq1}
P ( \tilde{v} ; x) \equiv P_0 \;\:,\; \textrm{in} \;\: B_{\delta}(x_0)
\end{equation}
So $ U $ is open and by connectedness $ U = \mathbb{R}^n $.

On the other hand, since $ \tilde{v} $ is bounded it holds that there exist a sequence $ (y_l)_{l \in \mathbb{N}} ) $ in $ \mathbb{R}^n $ such that
\begin{equation}\label{ProofThmGrBHigherOrdEq2}
\lim_{l \rightarrow + \infty} \Delta \tilde{v}(y_l) =0
\end{equation}

By the boundedness of $ \tilde{v} $ and $ \nabla \tilde{v} $, there exist subsequence still denoted by $ y_l $ such that
$ \tilde{v}(y_l) \rightarrow \tilde{v}_0 $ and $ \nabla \tilde{v}(y_l) \rightarrow \tilde{q}_0 $ and so we obtain
\begin{equation}\label{ProofThmGrBHigherOrdEq3}
0 < P_0 = \lim_{l \rightarrow \infty} P(\tilde{v}(y_l) , \nabla \tilde{v}(y_l) , \Delta \tilde{v}(y_l) ) = P( v_0, q_0 , 0) \leq 0
\end{equation}
and we conclude.
\end{proof}
$ \\ $

\begin{remark}\label{RmkHigherOrderPfunction}
The arguments above can be applied for $ P- $functions of various different forms. For example, we can consider $ P= G(u, \nabla u, ..., \nabla^{m-1} u) $ where $ G $ is a continuous function that satisfies $ G \leq 0 $ in the set where either $ \nabla u =0 $ or $ \Delta^k u =0 $ holds.
\end{remark}

$ \\ $

We now provide the generalization of Theorem \ref{theoremRigidityForFullyNonlinear} in the higher order case.
$ \\ $
\begin{theorem}\label{ThmLiouvilleForHigherOrderPDEs}
Let $ u $ be an entire solution of
\begin{equation}\label{ThmLiouvilleForHigherOrderGeneralEq}
F(u, \nabla u,..., \nabla^m u) =0 
\end{equation}
and let $  P(u,\nabla u, \Delta u,..., \Delta^{m-2}u) =P(u;x) $ be a $ P- $function of \eqref{ThmLiouvilleForHigherOrderGeneralEq} such that
\begin{equation}\label{ThmLiouvilleForHigherOrderPfunctionform}
\begin{gathered}
P = H( \Delta^k u) \;\;,\; \textrm{where} \;\: H \: : \mathbb{R} \rightarrow [0, +\infty) \;\: \textrm{and} \;\: H(t) > 0 \;,\; \textrm{if} \;\: t \neq 0.
\end{gathered}
\end{equation}
Suppose the assumptions of Theorem \ref{ThmBoundForHigherOrderPDEs} in case (ii).

Then $ \Delta^{k-1} u $ is a constant.
\end{theorem}
\begin{proof}
The proof is direct consequence of Theorem \ref{ThmBoundForHigherOrderPDEs} since $ P=H( \Delta^k u) = 0 $ which gives $ \Delta^k u \equiv 0 $ in $ \mathbb{R}^n $. Note that if $ k=1 $, the notation $ \Delta^0 u $ is used to denote $ u $.
\end{proof}

$ \\ $

\subsection{Local and Global Pointwise estimates}

The arguments of section 5, can be applied for higher order nonlinear equations. In this case, we extract local and global estimates for higher order of derivatives of $ u $, such as for the Laplacian. Throughout this subsection we will assume the stability assumption in Theorem \ref{ThmBoundForHigherOrderPDEs}, that is, the set of translations of solutions is closed in the $ C^{m-1}_{loc} $ norm. $ \\ $

\begin{proposition}\label{PropHarnackforHigherorder}
Let $ u : \Omega \rightarrow \mathbb{R} $ be a smooth and convex subsolution of 
\begin{equation}\label{PropHarnackForDelta^2u=..Equation}
\Delta^2 u - F(x,u,\nabla u, \nabla^2 u, \nabla^3 u) =0 \;\;,\; \textrm{with} \;\: F \geq 0
\end{equation}

Then for any $ B_R \subset \Omega $, any $ 0 < r<R $ and any $ p \geq 1 $,
\begin{equation}\label{PropHarnackForDelta^2u=..Statement}
\sup_{B_r} ( \Delta u)^2 \leq \frac{C}{(R-r)^{n/p}} || \Delta u||^2_{L^{2p}(B_R)}
\end{equation}
where $ C=C(n,p) $.

In addition, if $ \Omega = \mathbb{R}^n $ and $ u $ satisfies \eqref{AssumptionHigherOrderPDEs} with $ m=4 $, then $ u $ is a constant
\end{proposition}
\begin{proof}
\begin{equation}\label{ProofPropHarnackforHigherorderEq1}
\begin{gathered}
P_{x_i} = 2 \Delta u \Delta u_{x_i} \\
P_{x_i x_i} = 2 ( \Delta u_{x_i})^2 + 2 \Delta u \Delta u_{x_i x_i} \\
\Rightarrow \Delta P = 2| \nabla \Delta u |^2 + 2 \Delta u \Delta^2 u \geq 2| \nabla \Delta u |^2 + F \Delta u \geq 0
\end{gathered}
\end{equation}
Therefore, by classical estimates for subsolutions of elliptic equations he conclude (see for example Theorem 4.14 in \cite{HL}).

The second conclusion follows from Theorem \ref{ThmLiouvilleForHigherOrderPDEs}.
\end{proof}

$ \\ $

Furthermore, we give some examples of $ P -$functions of the form $ P= P(u, |\nabla u|, \Delta u) $ related to forth order nonlinear equations together with applications. $ \\ $

\begin{proposition}\label{P-functionsforHigherOrder}
Let $ u $ be a smooth solution of
\begin{equation}\label{ForthOrderPDE}
\begin{gathered}
a( \Delta u) [ | \nabla u |^2 \Delta^2 u - \Delta u ( \nabla u \cdot \nabla \Delta u)] = b(u) | \nabla u |^4 \\ \textrm{where} \;\: a,b \: : \mathbb{R} \rightarrow \mathbb{R} \;\: \textrm{and} \;\: a>0 \;,\; a' \geq 0
\end{gathered}
\end{equation}
and set $ P(s,t) = A(t) -B(s) $ such that $ A'=a $ and $ B''=b $.

Then $ P= P(u, \Delta u) = A(\Delta u) - B(u) $ is a $ P- $function of \eqref{ForthOrderPDE}.
$ \\ $
In addition, if $ u $ satisfies \eqref{AssumptionHigherOrderPDEs} with $ m=4 \;,\: A(0)=0 \;,\; B(u) \geq 0 $ and assume that for any $ v \in C^2_{loc}(\mathbb{R}^n) $ translation limit with $ \nabla v (0) =0 $ we have $ \Delta v (0) \leq 0 $, then
\begin{equation}\label{PropP-functionHigherOrderLaplacianBound}
\Delta u \leq \Gamma (u) \;\;\: \forall \: x \in \mathbb{R}^n, \;\: \textrm{where} \;\: \Gamma (u) = A^{-1}(B(u)) .
\end{equation}
\end{proposition}
\begin{proof}
We have
\begin{equation}\label{P-functionsforHigherOrderproofeq1}
P_{x_i} = P_s u_{x_i} + P_t \Delta u_{x_i}
\end{equation}
and so,
\begin{equation}\label{P-functionsforHigherOrderproofeq2}
\begin{gathered}
\Delta u (\nabla P \cdot \nabla u) = P_s | \nabla u |^2 \Delta u + P_t \Delta u \sum_{i=1}^n u_{x_i} \Delta u_{x_i} \\ \Leftrightarrow
- B'(u) | \nabla u |^2 \Delta u = \Delta u (\nabla P \cdot \nabla u) - A'(\Delta u)\Delta u \sum_{i=1}^n u_{x_i} \Delta u_{x_i}  
\end{gathered}
\end{equation}
on the other hand we have
\begin{equation}\label{P-functionsforHigherOrderproofeq3}
\begin{gathered}
P_{x_i x_i} = P_{ss} u_{x_i}^2 + 2 P_{st} u_{x_i} \Delta u_{x_i} + P_{tt} ( \Delta u_{x_i})^2 + P_s u_{x_i x_i} + P_t \Delta u_{x_i x_i} \\
\Rightarrow  \Delta P = (-B''(u)) | \nabla u |^2 + A''( \Delta u) \sum_{i=1}^n (\Delta u_{x_i})^2 - B'(u) \Delta u + A'( \Delta u) \Delta^2 u
\end{gathered}
\end{equation}
and by \eqref{P-functionsforHigherOrderproofeq2} and the assumptions of $ A $ and $ B $, \eqref{P-functionsforHigherOrderproofeq3} becomes
\begin{equation}\label{P-functionsforHigherOrderproofeq4}
| \nabla u |^2 \Delta P - \Delta u (\nabla P \cdot \nabla u) \geq a( \Delta u) [ | \nabla u |^2 \Delta^2 u - \Delta u ( \nabla u \cdot \nabla \Delta u)] - b(u) | \nabla u |^4 =0 
\end{equation}

For the bound of the Laplacian, we have $ P(u,0) = - B(u) \leq 0 $ and $ \mu = | \nabla u |^2 >0 \;\: \forall \: x \in \mathbb{R}^n $ since for any $ v \in C^2_{loc}(\mathbb{R}^n) $ translation limit with $ \nabla v (0) =0 $ we have $ \Delta v (0) \leq 0 $, the assumption (i) in Theorem \ref{ThmBoundForHigherOrderPDEs} is satisfied and we conclude.
\end{proof}

$ \\ $
\begin{proposition}\label{P-functionFor|Hes u|^2=..} Let $ u $ be a smooth solution of
\begin{equation}\label{|Hes u|^2=..}
\begin{gathered}
| \nabla^2 u|^2 = F(u, | \nabla u |^2, \Delta u) + \frac{u}{2} \Delta^2 u \\
\textrm{where} \;\: F : \mathbb{R}^3 \rightarrow \mathbb{R} \; \textrm{is such that} \;\: F(s,t,w) \geq \frac{1}{2}w^2.
\end{gathered}
\end{equation}

Then $ P=P(u, | \nabla u|^2 , \Delta u) = | \nabla u |^2 -u \Delta u $ is a $ P- $function of \eqref{|Hes u|^2=..}.
$ \\ $
In addition, if $ u $ is non negative, convex solution of \eqref{|Hes u|^2=..} that satisfies assumption \eqref{AssumptionHigherOrderPDEs}, then
\begin{equation}\label{P-functionFor|Hes u|^2=.GradientBound}
| \nabla u |^2 \leq u \Delta u \;\;\;,\; \forall \; x \in \mathbb{R}^n.
\end{equation}
\end{proposition}
\begin{proof}
We have that
\begin{align*}
P_{x_i} = 2 \sum_{j=1}^n u_{x_j} u_{x_j x_i} - u_{x_i} \Delta u - u \Delta u_{x_i}
\end{align*}
and
\begin{align*}
\Delta P = 2 | \nabla^2 u|^2 + 2 \nabla u \nabla \Delta u - (\Delta u)^2 - 2 \nabla u \nabla \Delta u - u \Delta^2 u
\end{align*}
so by \eqref{|Hes u|^2=..},
\begin{align*}
\Delta P = 2 F(u, | \nabla u |^2, \Delta u) - ( \Delta u)^2 \geq 0
\end{align*}

For the gradient bound we see that $ P(u,0,\Delta u) = - u \Delta u \leq 0 $ since $ u $ is non negative and convex, so the assumption (i) of Theorem \ref{ThmBoundForHigherOrderPDEs} is satisfied and we conclude.
\end{proof}

As a result, we have the following pointwise estimate $ \\ $

\begin{corollary}\label{CorPointwiseEstimateforHigherOrderPDE} Let $ u : B_2 \subset \mathbb{R}^n \rightarrow \mathbb{R} $ be a smooth solution of
\begin{equation}\label{|Hes u|^2=..}
\begin{gathered}
| \nabla^2 u|^2 = F(u, | \nabla u |^2, \Delta u) + \frac{u}{2} \Delta^2 u \\
\textrm{where} \;\: F : \mathbb{R}^3 \rightarrow \mathbb{R} \; \textrm{is such that} \;\: F(s,t,w) \geq \frac{1}{2}w^2.
\end{gathered}
\end{equation}

Then
\begin{equation}\label{HigherOrderPointwiseEstimate}
\begin{gathered}
| \nabla u(x) |^2 - u(x) \Delta u(x) \leq C ( || u||^2_{H^1(B_2)} + || \Delta u||^2_{L^2(B_2)}) \;, \\ \forall \: x \in B_1 = \lbrace y \in \mathbb{R}^n \: : \: | y | <1 \rbrace \;,\: \textrm{and} \;\: C \;\: \textrm{depends only on} \;\: n.
\end{gathered}
\end{equation}
\end{corollary}
\begin{proof}
By Proposition \ref{P-functionFor|Hes u|^2=..}, we have that $ P= | \nabla u |^2 -u \Delta u =P(u;x) $ is subharmonic. Therefore we have
\begin{equation}\label{CorPointwiseEstimateHigherOrderProofEq1}
P(u;x) \leq \frac{1}{| B(x,r) |} \int_{B(x,r)} P(u;y) dy \;\;,\; \forall \; B(x,r) \subset B_2
\end{equation}
Also, $ P \leq | \nabla u |^2 + \frac{1}{2}(u^2 + (\Delta u)^2) $. 

So, 
\begin{equation}\label{CorPointwiseEstimateHigherOrderProofEq2}
\int_{B(x,r)} P(u;y) dy \leq || u||^2_{H^1(B_2)} + || \Delta u||^2_{L^2(B_2)} \;,\: \forall \; B(x,r) \subset B_2
\end{equation}
Thus, for any $ x \in B_1 $ (since $ B(x,1) \subset B_2 $), we have
\begin{equation}
P(u;x) \leq \frac{1}{| B_1 |} ( || u||^2_{H^1(B_2)} + || \Delta u||^2_{L^2(B_2)}) 
\end{equation}
\end{proof}

A special case of Corollary \ref{CorPointwiseEstimateforHigherOrderPDE} is the following estimate $ \\ $

\begin{corollary}\label{CorPointwiseEstimateforHigherOrderPDEinthePlane}
Let $ u : B_2 \subset \mathbb{R}^2 \rightarrow \mathbb{R} $ be a smooth solution of
\begin{equation}\label{det(nabla^2 u)=-u/4Delta^2u..}
\begin{gathered}
det( \nabla^2 u) + G(u, | \nabla u |^2) + \frac{u}{4} \Delta^2 u =0 \\
\textrm{where} \;\: G : \mathbb{R}^2 \rightarrow [0 , + \infty)
\end{gathered}
\end{equation}
Then
\begin{equation}\label{CorHigherOrderPEstinthePlane}
| \nabla u(x) |^2 - u(x) \Delta u(x) \leq C ( || u||^2_{H^1(B_2)} + || \Delta u||^2_{L^2(B_2)})
\end{equation}
\end{corollary}
$ \\ $
\begin{proof}
We write $ det( \nabla^2 u) = \dfrac{1}{2}( ( \Delta u)^2 - | \nabla^2 u |^2) $ since $ u $ is defined a domain in the plane and then the proof is a consequence of Corollary \ref{CorPointwiseEstimateforHigherOrderPDE} for $ F= 2 G(u, | \nabla u |^2) + ( \Delta u)^2 \;\:,\; G \geq 0 $.
\end{proof}
$ \\ $

\subsection{A Liouville theorem and a De Giorgi-type property}

A direct consequence of Theorem \ref{ThmLiouvilleForHigherOrderPDEs} is the following $ \\ $

\begin{corollary}\label{CorLiouvilleForDelta^2=c|Hes u|^2}
Let $ u $ be a convex entire subsolution of
\begin{equation}\label{CorLiouvilleForDelta^2u=c..Equation}
\sum_{i,j} a_{ij}(x,u, \nabla u ,\nabla^2 u,\nabla^3 u) \Delta u_{x_i x_j} - F(x,u,\nabla u, \nabla^2 u, \nabla^3 u) =0
\end{equation}
that satisfies assumption \eqref{AssumptionHigherOrderPDEs} with $ m=4 $ and assume $ a_{ij} $ satisfy the ellipticity condition \eqref{EllipticityConditionForFullyNonlinearElEq} and $ F \geq 0 $.

Then $ u $ is constant.
\end{corollary}
$ \\ $
\begin{proof}
Consider $ P =P(u,\nabla u, \Delta u) = ( \Delta u)^2 $, so as in the proof of Proposition \ref{PropHarnackforHigherorder} we calculate
\begin{equation}\label{ProofCorLiouvilleForDelta^2u=c..Eq1}
\begin{gathered}
P_{x_i x_j} = 2 \Delta u_{x_i} \Delta u_{x_j} + 2 \Delta u \Delta u_{x_i x_j} \\
\Rightarrow \sum_{i,j} a_{ij} P_{x_i x_j} \geq 2 \lambda | \nabla \Delta u |^2 + F \Delta u \geq 0
\end{gathered}
\end{equation}
and $ P(u,\nabla u, 0) = 0 $ with $ \mu =1 $, so by Theorem \ref{ThmLiouvilleForHigherOrderPDEs} we obtain $ \Delta u \equiv 0 $ in $ \mathbb{R}^n $ and $ u $ is bounded by \eqref{AssumptionHigherOrderPDEs}, so $ u $ is constant.
\end{proof}
$ \\ $

Finally, we have a De Giorgi-type property
$ \\ $

\begin{proposition}\label{AbstractThmDeGiorgiHigherOrder}
Let $ u : \mathbb{R}^2 \rightarrow \mathbb{R} $ be a smooth and bounded solution of
\begin{equation}\label{AbstractThmDeGiorgiHigherOrderGeneralEquation}
F(u, \nabla u, \nabla^2 u, \nabla^3 u, \nabla^4 u) =0
\end{equation}
such that $ u_{y} >0 $ and assume $ P=P(u, \Delta u) = A( \Delta u) + B(u) $ is a $ P -$function of \eqref{AbstractThmDeGiorgiHigherOrderGeneralEquation}, satisfying $ P(u,0) \:,\; B(u) \leq 0 \;,\: \mu >0 $ and $ P_t>0 $.

If there exists $ x_0 \in \mathbb{R}^2 $ such that
\begin{equation}\label{AbstractThmDeGiorgiHigherOrderPattainsSup}
0 \leq P(u(x_0), \Delta u(x_0)) = \sup_{\mathbb{R}^n} P(u, \Delta u) < + \infty
\end{equation}
then there exists a function $ g: \mathbb{R} \rightarrow \mathbb{R} $ such that
\begin{equation}\label{AbstractThmDeGiorgiHigherOrderuonedim}
u(x) = g( a x +b y ) \;\;\;,\;\: \textrm{for} \;\: a,b \in \mathbb{R}
\end{equation}
\end{proposition}
\begin{proof}
Arguing as in the proof of Theorem \ref{ThmDeGiorgitypeforDeltau=F} we obtain that
\begin{equation}\label{AbstractThmDeGiorgiHigherOrderPconst}
P (u, \Delta u) \equiv c_0 \;\;\;,\;\; \textrm{where} \;\: c_0 = \sup_{\mathbb{R}^n} P(u, \Delta u)
\end{equation}
since $ P_t>0 $ we have
\begin{equation}\label{AbstractThmDeGiorgiHigherOrderAllenCahn}
\Delta u = f(u) \;\;\;,\;\; \textrm{for some} \;\: f: \mathbb{R} \rightarrow \mathbb{R}
\end{equation}
and $ u $ is bounded entire solution of \eqref{AbstractThmDeGiorgiHigherOrderAllenCahn} such that $ u_y >0 $. 

Therefore, by Theorem 1.1 in \cite{GG}, we conclude that
\begin{equation}\label{proofThmDeGiorgiHigherOrderuonedim}
u(x) = g( a x +b y) \;\;\;,\;\: \textrm{for some} \;\: g: \mathbb{R} \rightarrow \mathbb{R}
\end{equation}
\end{proof}


\textbf{Acknowledgments:} I wish to thank my advisor professors N. Alikakos and C. Makridakis for their support. Also, I would like to thank professors N. Alikakos and A. Farina for their useful suggestions that lead to various improvements. Finally, I would like to thank professor C. Gui for both the advises and for his interest in this work. 

$ \\ $


\begin{thebibliography}{99}
\itemsep=0pt

\bibitem{AA} Ambrosio L., Cabre, X., \textit{Entire solutions of semilinear elliptic equations in $ \mathbb{R}^3 $ and a conjecture of De Giorgi}, J. Amer. Math. Soc. 13 (2000), 725-739.


\bibitem{AAC} Alberti,G., Ambrosio L., Cabre, X. \textit{On a Lon-Standing Conjecture of E.De Giorgi Symmetry in 3D for General Nonlinearities and a Local Minimality Property}, Acta Appl. Math.
(2001). https://doi.org/10.1023/A:1010602715526

\bibitem{ACM} Ambrosio L., Carlotto A., Massaccesi A. \textit{Lectures on Elliptic Partial Differential Equations}, Publications of the Scuola Normale Superiore (2018).
https://doi.org/10.1007/978-88-7642-651-3

\bibitem{BG} Banerjee A., Garofalo N., \textit{Modica type gradient estimates for an inhomogeneous variant of the normalized $ p$-Laplacian evolution}, Nonlinear Analysis: Theory, Methods and Applications, Volume 121, July 2015, Pages 458-468. https://doi.org/10.1016/j.na.2015.02.003

\bibitem{CaffarelliCabre} Caffarelli L., Cabré X. \textit{Fully Nonlinear Elliptic Equations}, American Mathematical Society , Vol 43 (1995).

\bibitem{CC} Caffarelli, L., Crandall, M. \textit{Distance Functions and Almost Global Solutions of Eikonal Equations}, Com. Partial Dif. Eq., 
(2010). https://doi.org/10.1080/03605300903253927


\bibitem{CGS} Caffarelli, L., Garofalo, N., Segala, F. \textit{A Gradient Bound for Entire Solutions of Quasi-Linear Equations and Its Consequences}, Comm. Pure and Applied Math.
(1994). https://doi.org/10.1002/cpa.3160471103

\bibitem{CFV} Castellaneta D., Farina A., Valdinoci E., \textit{A pointwise gradient estimate for solutions of singular and degenerate PDEs in possibly unbounded domains with nonnegative mean curvature}, Commun. Pure Appl. Anal, 2012.

\bibitem{CDFGV} Cavaterra, C., Dipierro, S., Farina, A., Gao, Z., Valdinoci, E.
\textit{Pointwise gradient bounds for entire solutions of elliptic equations with non-standard growth conditions and general nonlinearities}, Journal of Dif. Eq.
(2021). https://doi.org/10.1016/j.jde.2020.08.007

\bibitem{CFV} Cozzi M., Farina A., Valdinoci E., \textit{Gradient Bounds and Rigidity Results for Singular, Degenerate, Anisotropic Partial Differential Equations}. Commun. Math. Phys. 331, 189–214 (2014). 
https://doi.org/10.1007/s00220-014-2107-9

\bibitem{DG} Danielli, D., Garofalo, N. \textit{Properties of entire solutions of non uniformly elliptic equations arising in geometry and in phase transitions}, Calc. of Var.
(2002). https://doi.org/10.1007/s005260100133

\bibitem{D} Dupaigne, L., \textit{Stable solutions of Elliptic Partial Differential Equations}, Chapman and Hall/CRC Monographs and Surveys in Pure and Applied Mathematics, 143 (2011).


\bibitem{FV} Farina, A., Valdinoci, E. \textit{Gradient bounds for anisotropic partial differential equations}, Calc. of Var.
(2013). https://doi.org/10.1007/s00526-013-0605-9


\bibitem{FV2} Farina, A., Valdinoci, E. \textit{Pointwise estimates and rigidity results  for entire solutions of nonlinear elliptic pde's}, 	
ESAIM: COCV
(2013). https://doi.org/10.1051/cocv/2012024

\bibitem{FV3} Farina, A., Valdinoci, E. \textit{A pointwise gradient estimate in possibly unbounded domains with nonnegative mean curvature}, Adv. in Math.
(2010). https://doi.org/10.1016/j.aim.2010.05.008

\bibitem{FV4} Farina, A., Valdinoci, E. \textit{A pointwise gradient bound for elliptic equations on compact manifolds with nonnegative Ricci curvature}, Discrete Contin. Dyn. Syst, 2011.

\bibitem{FWX} Fazly M., Wei J., Xu X., \textit{A pointwise inequality for the fourth-order Lane–Emden equation}, Analysis and PDE, Vol. 8 (2015), No. 7, 1541–1563. 
DOI: 10.2140/apde.2015.8.1541

\bibitem{GL} Garofalo N., Lewis, J. \textit{A Symmetry Result Related to Some Overdetermined Boundary Value Problems}, American Journal of Mathematics, Vol. 111, No. 1, 1989, pp. 9-33.
https://doi.org/10.2307/2374477

\bibitem{GS} Garofalo N., Sartori E., \textit{Symmetry in exterior boundary value problems for quasilinear elliptic equations via blow-up and a priori estimates}, Adv. Differential Equations 4 (2) 137 - 161, 1999. https://doi.org/10.57262/ade/1366291411


\bibitem{G} Gazoulis, D., \textit{A Relation of the Allen-Cahn equations and the Euler equations and applications of the equipartition}. Nonlinear Differ. Equ. Appl. 
(2023). https://doi.org/10.1007/s00030-023-00888-2

\bibitem{GZ} Gazoulis, D., Zacharopoulos, G., \textit{Gradient bounds and Liouville theorems for quasi-linear equations on compact manifolds with nonnegative Ricci curvature}, Proc. Amer. Math. Soc. 153 (2025), 4501-4516

\bibitem{GG} Ghoussoub, N., Gui, C. \textit{On a conjecture of De Giorgi and some related problems}, Math. Ann. 
(1998). https://doi.org/10.1007/s002080050196

\bibitem{GH} Giaquinta, M., Hildebrandt, S. \textit{Calculus of Variations I}, Berlin: Springer-Verlag, Heidelberg (1996).


\bibitem{GT} Gilbarg,D., Trudinger,N., S. \textit{Elliptic Partial Differential Equations of Second Order}, Berlin: Springer-Verlag, reprint of 1998 edition (2001).

\bibitem{Gui} Gui C., \textit{Hamiltonian identities for elliptic partial differential equations}, Journal of Functional Analysis, Volume 254, Issue 4, 15 February 2008, Pages 904-933.
https://doi.org/10.1016/j.jfa.2007.10.015

\bibitem{HL} Han, Q., Lin, F. \textit{Elliptic Partial Differential Equations}, AMS, Courant Institute of Mathematical Sciences, Second Ed. (2011).


\bibitem{Modica} Modica, L. \textit{A Gradient Bound and a Liouville Theorem for Nonlinear Poisson Equations}, Comm. Pure and App. Math.
(1985). https://doi.org/10.1002/cpa.3160380515

\bibitem{PP} Payne, L., E., Philippin, G., A. \textit{On Maximum Principles for a Class of Nonlinear Second-Order Elliptic Equations}, Journal Dif. Eq.
(1980). https://doi.org/10.1016/0022-0396(80)90086-8

\bibitem{PP2} Payne, L., E., Philippin, G., A. \textit{Some Maximum Principles for Nonlinear Elliptic Equations in divergence form with applications to capillary surfaces and to surfaces of constant mean curvature}, Nonlinear Anal., Theory, Methods \& Appl.
(1979). https://doi.org/10.1016/0362-546X(79)90076-2

\bibitem{PW} M. H. Protter, H. F. Weinberger, \textit{Maximum Principles in Differential Equations}
, Springer-Verlag New York, Inc. 1984.
https://doi.org/10.1007/978-1-4612-5282-5 

\bibitem{Serrin} Serrin, J. \textit{Entire solutions of nonlinear Poisson equations}, Proc. London Math. Soc. 
(1972). https://doi.org/10.1112/plms/s3-24.2.348

\bibitem{PS} Pucci P, Serrin J., \textit{The Maximum Principle},
Progress in Nonlinear Differential Equations and Their Applications, vol. 73, Birkhäuser Verlag, Basel (2007)

\bibitem{Sperb} Sperb, R., P. \textit{Maximum Principles and their Applications}, Zurich, Switzerland, Academic Press (1981).


\bibitem{Smyrnelis} Smyrnelis, P. \textit{Gradient estimates for semilinear elliptic systems and other related results}, Proc. Royal Soc. Edinburgh
(2015). https://doi.org/10.1017/S0308210515000347


\bibitem{T} Tolksdorf, P. \textit{Regularity for a More General Class of Quasilinear Elliptic Equations}, J. Dif. Eq.
(1984). https://doi.org/10.1016/0022-0396(84)90105-0

\bibitem{Trudinger} Trudinger, N. S. \textit{Fully nonlinear, uniformly elliptic equations under natural structure conditions}, Trans. Amer. Math. Soc. 
(1983). https://doi.org/10.2307/1999182

\end{thebibliography}
\end{document}